\documentclass[10pt]{article}
\usepackage{etex}

\usepackage{lmodern}
\usepackage[T1]{fontenc}
\usepackage{amsmath}
\usepackage{amsthm}
\usepackage{amssymb}
\usepackage{mathabx} % For \bigominus
\usepackage{url}
\usepackage{latexsym}
\usepackage{titlefoot}
\usepackage[small]{titlesec}
\usepackage{units} % for \nicefrac
\usepackage[small,it]{caption}
\usepackage[square,comma,numbers,sort&compress]{natbib}

\usepackage{tikz}
\usetikzlibrary{arrows, automata, decorations.pathreplacing, fit, matrix, patterns, positioning}
\usepackage{tikz-qtree}
\usepackage{xifthen} % For if-then control (used in a Tikz macro).

\usepackage{color}
\definecolor{lightgray}{rgb}{0.8, 0.8, 0.8}
\definecolor{darkgray}{rgb}{0.7, 0.7, 0.7}

\usepackage[bookmarks]{hyperref}
\hypersetup{
	colorlinks=true,
	linkcolor=black,
	anchorcolor=black,
	citecolor=black,
	urlcolor=black,
	pdfpagemode=UseThumbs,
	pdftitle={Growth rates of permutation classes: categorization up to the uncountability threshold},
	pdfsubject={Combinatorics},
	pdfauthor={Jay Pantone and Vincent Vatter},
}

\newcounter{todocounter}

\newcommand{\minisec}[1]{\noindent{\sc #1.}}

\theoremstyle{plain}
\newtheorem{theorem}{Theorem}[section]
\newtheorem{proposition}[theorem]{Proposition}

\newtheorem*{egf*}{Exponential Growth Formula}
\newtheorem*{thm-xi-grs-main*}{Theorem~\ref{thm-xi-grs-main}}
\theoremstyle{definition}

\makeatletter
\newfont{\footsc}{cmcsc10 at 8truept}
\newfont{\footbf}{cmbx10 at 8truept}
\newfont{\footrm}{cmr10 at 10truept}
\renewenvironment{abstract}{
	\begin{list}{}%
	{\setlength{\rightmargin}{1in}%
	\setlength{\leftmargin}{1in}}%
	\item[]\ignorespaces\begin{small}}%
	{\end{small}\unskip\end{list}%
}

\newcommand{\Av}{\operatorname{Av}}

\newcommand{\Sub}{\operatorname{Sub}}
\newcommand{\PropSub}{\Sub^{<}}

\newcommand{\C}{\mathcal{C}}

\renewcommand{\O}{\mathcal{O}}

\newcommand{\gr}{\mathrm{gr}}
\newcommand{\lgr}{\underline{\gr}}
\newcommand{\ugr}{\overline{\gr}}

\newcommand{\st}{\::\:}

\newcommand{\ml}{\operatorname{ml}}
\newcommand\mybullet{\raisebox{-5pt}{\normalsize \ensuremath{\bullet}}}
\newcommand\mycirc{\raisebox{-5pt}{\normalsize \ensuremath{\circ}}}

\makeatletter
\def\absdot{\@ifnextchar[{\@absdotlabel}{\@absdotnolabel}}
	\def\@absdotlabel[#1]#2{%
		\node at #2 {\normalsize \mybullet};
		\node at #2 [below=2pt] {\ensuremath{#1}};
	}
	\def\@absdotnolabel#1{%
		\node at #1 {\normalsize \mybullet};
	}
\def\absdothollow{\@ifnextchar[{\@absdothollowlabel}{\@absdothollownolabel}}
	\def\@absdothollowlabel[#1]#2{%
		\node at #2 {\normalsize \textcolor{white}{\mybullet}};
		\node at #2 {\normalsize \mycirc};
		\node at #2 [below=2pt] {\ensuremath{#1}};
	}
	\def\@absdothollownolabel#1{%
		\node at #1 {\normalsize \textcolor{white}{\mybullet}};
		\node at #1 {\normalsize \mycirc};
	}
\makeatother

\newcommand{\plotperm}[1]{
	\foreach \j [count=\i] in {#1} {
		\absdot{(\i,\j)};
	};
}

\newcommand{\plotpartialperm}[1]{
	\foreach \i/\j in {#1} {
		\absdot{(\i,\j)};
	};
}

\newcommand{\plotpermbox}[4]{
	\draw [darkgray, very thick, rounded corners=0.01, line cap=round]
		({#1-0.5}, {#2-0.5}) rectangle ({#3+0.5}, {#4+0.5});
}

\newcommand{\plotpermgraph}[1]{
	\foreach \j [count=\i] in {#1} {
		\foreach \b [count=\a] in {#1} {
			\ifthenelse{\a<\i \AND \b>\j}{\draw (\a,\b)--(\i,\j);}{}
		};
	};
	\plotperm{#1};
}

\newcommand{\plotpermdyckpath}[1]{
	\draw [ultra thick, rounded corners=0.01, line cap=round] (0.5,0.5)
	\foreach \step in {#1} {
		\ifnum\step=1
			-- ++(0,1)
		\else
			-- ++(1,0)
		\fi
	};
}

\newcommand{\matrixpermwithzeros}[2]{
	\foreach \y [count=\x] in {#2} {
		\foreach \j in {1, 2, ..., #1} {
			\ifthenelse{\j=\y}{
				\node at (\x,\j) {\tiny $\mathbf{1}$};
			}{
				\node at (\x,\j) {\textcolor{darkgray}{\tiny $0$}};
			}
		}
	}
}

\newcommand{\plotdyckpath}[1]{
	\draw[ultra thick, line cap=round] (0.5,0)
	\foreach \step in {#1} {
		\ifnum\step=1
			-- ++(1,1)
		\else
			-- ++(1,-1)
		\fi
	};
}

\newcommand{\arcskinnyplain}[2]{
	\draw (#1,0) arc (180:0:{(#2-#1)/2});
}

\newcommand{\matchsmall}[1]{
	\begin{tikzpicture}[scale=.1, anchor=base]
		\def\h{0};
		\def\maxh{0};
		\foreach \i/\j in {#1} {
			\pgfmathparse{\j-\i};
			\let\h\pgfmathresult;
			\pgfmathifthenelse{\h>\maxh}{\h}{\maxh};
			\global\let\maxh\pgfmathresult;
		};
		\pgftransformyscale{{4.5/\maxh}};
		\foreach \i/\j in {#1} {
			\arcskinnyplain{\i}{\j};
		};
	\end{tikzpicture}
}

\newcommand{\matchpermsmall}[1]{
	\begin{tikzpicture}[scale=.1, anchor=base]
		\foreach \j [count=\n] in {#1} {};
		\def\h{0};
		\def\maxh{0};
		\foreach \j [count=\i] in {#1} {
			\pgfmathparse{2*\n+1-\j-\i};
			\let\h\pgfmathresult;
			\pgfmathifthenelse{\h>\maxh}{\h}{\maxh};
			\global\let\maxh\pgfmathresult;
		};
		\pgftransformyscale{{4.5/\maxh}};
		\foreach \j [count=\i] in {#1} {
			\arcskinnyplain{\i}{{2*\n+1-\j}};
		};
	\end{tikzpicture}
}

\newcommand{\plotpinsequence}[1]{
	\absdot{(0,0)}{};
	\edef\n{0}
	\edef\s{0}
	\edef\e{0}
	\edef\w{0}
	\edef\x{0}
	\edef\y{0}
	\foreach \pin [remember=\pin as \oldpin (initially 1), count=\i] in {#1} {
		\ifthenelse{\pin=1 \OR \pin=2}{%up or down
			\ifthenelse{\oldpin=3}{% previous=right
				\xdef\x{\number\numexpr\e-1}
			}{
				\xdef\x{\number\numexpr\w+1}
			}
			\ifnum\i=1 %expand eastern box by 1 if 1st pin
				\pgfmathparse{\e+1}
 				\xdef\e{\pgfmathresult}
			\fi	
		}{ %left or right
			\ifthenelse{\oldpin=1}{% previous=up
				\xdef\y{\number\numexpr\n-1}
			}{
				\xdef\y{\number\numexpr\s+1}
			}
			\ifnum\i=1 %expand southern boundary by 1 if 1st pin
				\pgfmathparse{\s-1}
 				\xdef\s{\pgfmathresult}
			\fi	
		}
		\ifnum\pin=1 %up
			\pgfmathparse{\n+2}
 			\xdef\n{\pgfmathresult}		
			\absdot{(\x,\n)}{};
			\ifnum\i>1
				\draw (\x,\n) -- (\x,\y-0.5);
			\else
			\fi
		\fi
		\ifnum\pin=2 % down		
			\pgfmathparse{\s-2}
 			\xdef\s{\pgfmathresult}
			\absdot{(\x,\s)}{};
			\ifnum\i>1
				\draw (\x,\s) -- (\x,\y+0.5);
			\else
			\fi
		\fi
		\ifnum\pin=3 %right
			\pgfmathparse{\e+2}
 			\xdef\e{\pgfmathresult}
			\absdot{(\e,\y)}{};
			\ifnum\i>1
				\draw (\e,\y) -- (\x-0.5,\y);
			\else
			\fi
		\fi
		\ifnum\pin=4 %left
			\pgfmathparse{\w-2}
 			\xdef\w{\pgfmathresult}
			\absdot{(\w,\y)}{};
			\ifnum\i>1
				\draw (\w,\y) -- (\x+0.5,\y);
			\else
			\fi
		\fi		
	};
}

\newcommand{\gridsmallhoriz}[1]{
	\begin{tikzpicture}[scale=1, anchor=base]
    \def\gridheight{1};
    \foreach \dir [count=\gridwidth] in {#1} {
    };
	  \pgftransformxscale{{0.225/\gridwidth}};
		\pgftransformyscale{{0.225/\gridheight}};
    \foreach \dir [count=\i] in {#1} {
      \ifthenelse{\dir>0}{
		    \draw [semithick, line cap=round] ({\i-1}, 0)--(\i, 1);
      }{
        \draw [semithick, line cap=round] ({\i-1}, 1)--(\i, 0);
      };
    };
  \end{tikzpicture}
}

\newcommand{\gridhoriz}[1]{
	\begin{tikzpicture}[scale=1, anchor=base]
	  \pgftransformxscale{0.225};
		\pgftransformyscale{0.225};
    \foreach \dir [count=\i] in {#1} {
      \ifthenelse{\dir>0}{
		    \draw [semithick, line cap=round] ({\i-1}, 0)--(\i, 1);
      }{
        \draw [semithick, line cap=round] ({\i-1}, 1)--(\i, 0);
      };
    };
  \end{tikzpicture}
}

\newcommand{\gridsmallvert}[1]{
	\begin{tikzpicture}[scale=1, anchor=base]
    \def\gridwidth{1};
    \foreach \dir [count=\gridheight] in {#1} {
    };
	  \pgftransformxscale{{0.225/\gridwidth}};
		\pgftransformyscale{{0.225/\gridheight}};
    \foreach \dir [count=\i] in {#1} {
      \ifthenelse{\dir>0}{
		    \draw [semithick, line cap=round] (0, {\i-1})--(1, \i);
      }{
        \draw [semithick, line cap=round] (0, \i)--(1, {\i-1});
      };
    };
  \end{tikzpicture}
}

\newcommand{\gridverysmallvert}[1]{
	\begin{tikzpicture}[scale=1, anchor=base]
    \def\gridwidth{1};
    \foreach \dir [count=\gridheight] in {#1} {
    };
	  \pgftransformxscale{{0.175/\gridwidth}};
		\pgftransformyscale{{0.175/\gridheight}};
    \foreach \dir [count=\i] in {#1} {
      \ifthenelse{\dir>0}{
		    \draw [semithick, line cap=round] (0, {\i-1})--(1, \i);
      }{
        \draw [semithick, line cap=round] (0, \i)--(1, {\i-1});
      };
    };
  \end{tikzpicture}
}

\datefoot{\today}
\amssubj{05A05, 05A16 (primary), 05A15 (secondary).}
\newpagestyle{main}[\small]{
	\headrule
	\sethead[\usepage][][]
	{\sc Growth Rates of Permutation Classes}{}{\usepage}
}

\title{\sc Growth Rates of Permutation Classes:\\ Categorization up to the Uncountability Threshold}

\author{%
\begin{tabular}{ccc}
Jay Pantone\footnotemark[1]
&\rule{3pt}{0pt}&
Vincent Vatter\footnote{Both authors were partially supported by the National Science Foundation under Grant Number DMS-1301692.}\\[-0.25ex]
\small Department of Mathematics, Statistics, and Computer Science
&&
\small Department of Mathematics\\[-0.5ex]
\small Marquette University
&&
\small University of Florida\\[-0.5ex]
\small Milwaukee, WI USA
&&
\small Gainesville, Florida USA\\[-1.5ex]
\end{tabular}
}

\titleformat{\section}{\large\sc}{\thesection.}{1em}{}
\date{}
 
\begin{document}
\maketitle

\pagestyle{main}

\begin{abstract}
In the antecedent paper to this it was established that there is an algebraic number $\xi\approx 2.30522$ such that while there are uncountably many growth rates of permutation classes arbitrarily close to $\xi$, there are only countably many less than $\xi$. Here we provide a complete characterization of the growth rates less than $\xi$. In particular, this classification establishes that $\xi$ is the least accumulation point from above of growth rates and that all growth rates less than or equal to $\xi$ are achieved by finitely based classes. A significant part of this classification is achieved via a reconstruction result for sum indecomposable permutations. We conclude by refuting a suggestion of Klazar, showing that $\xi$ is an accumulation point from above of growth rates of finitely based permutation classes.
\end{abstract}

\section{Introduction}

We are concerned here with the problem of determining the complete list of all growth rates of permutation classes. To be concrete, a \emph{permutation class} is a downset of permutations under the \emph{containment order}, in which $\sigma$ is \emph{contained} in $\pi$ if $\pi$ has a (not necessarily consecutive) subsequence that is \emph{order isomorphic} to $\sigma$ (i.e., has the same relative comparisons). If $\sigma$ is contained in $\pi$ then we write $\sigma\le\pi$; otherwise we say that $\pi$ \emph{avoids} $\sigma$.

Given a permutation class $\C$ we denote by $\C_n$ the subset of $\C$ consisting of its members of length $n$ (we think of permutations in one-line notation, so \emph{length} means the number of symbols). The Marcus--Tardos Theorem~\cite{marcus:excluded-permut:} (formerly the Stanley--Wilf Conjecture) shows that for every proper permutation class $\C$ (meaning, every class except that containing all permutations), the cardinalities $|\C_n|$ grow at most exponentially. Thus the \emph{upper} and \emph{lower growth rates} of the permutation class $\C$, defined respectively by
\[
	\ugr(\C) = \limsup_{n\rightarrow\infty}\sqrt[n]{|\C_n|}
	\quad\mbox{and}\quad
	\lgr(\C) = \liminf_{n\rightarrow\infty}\sqrt[n]{|\C_n|}
\]
are finite for every proper permutation class $\C$. When these two quantities are equal (which is conjectured to hold for all classes and is known to hold for all classes in this work) we denote their common value by $\gr(\C)$ and call it the \emph{growth rate} of $\C$. For a thorough introduction to permutation classes, including all of the notions reviewed here, we refer the reader to the second author's survey~\cite{vatter:permutation-cla:} in the \emph{CRC Handbook of Enumerative Combinatorics}.

Work on determining the set of growth rates of permutation classes has identified several notable phase transitions where both the set of growth rates and the corresponding permutation classes undergo dramatic changes:
\begin{itemize}
\item Kaiser and Klazar~\cite{kaiser:on-growth-rates:} showed that $2$ is the least accumulation point of growth rates and determined all growth rates below $2$. They also showed that every class of growth rate less than the golden ratio $\varphi$ has eventually polynomial enumeration, and thus growth rate $0$ or $1$.
\item Vatter~\cite{vatter:small-permutati:} established that there are uncountably many permutation classes of growth rate $\kappa\approx 2.20557$ (a specific algebraic integer), but only countably many of growth rate less than $\kappa$. In the same paper, all growth rates under $\kappa$ are characterized. This characterization shows that there is a sequence of accumulation points of growth rates which themselves accumulate at $\kappa$, making it the least second-order accumulation point of growth rates.
\item Albert, Ru\v{s}kuc, and Vatter~\cite{albert:inflations-of-g:} established that every class of growth rate less than $\kappa$ has a rational generating function, while there are (by an elementary counting argument) classes of growth rate $\kappa$ with non-rational (and even non-D-finite) generating functions.
\item Bevan~\cite{bevan:intervals-of-pe:}, refining work of Albert and Linton~\cite{albert:growing-at-a-pe:} and Vatter~\cite{vatter:permutation-cla:lambda:}, established that the set of growth rates contains every real number above $\lambda_B\approx 2.35698$ and that the set contains an infinite sequence of intervals whose infimum is $\theta_B\approx 2.355256$. (Both $\lambda_B$ and $\theta_B$ are specific algebraic integers.)
\end{itemize}

\begin{figure}
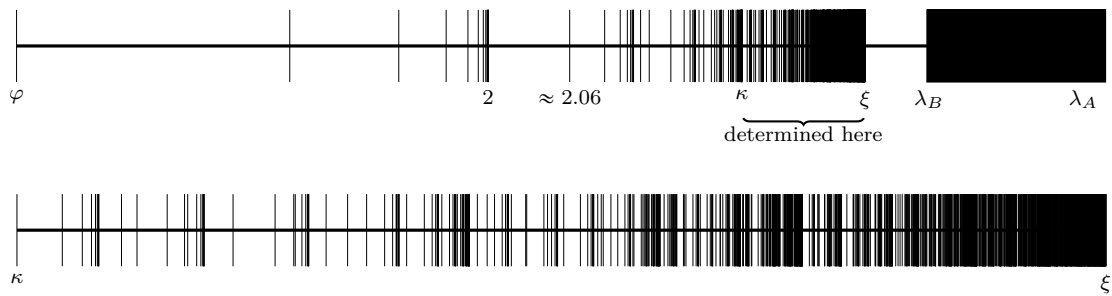

\begin{footnotesize}
\begin{center}
	\input{gr-graph-phi-to-2p5.tex}
	\\[10pt]
	\input{gr-graph-kappa-to-xi.tex}
\end{center}
\end{footnotesize}
\caption{The set of all growth rates of permutation classes between the golden ratio $\varphi$ and $2.5$, as presently known, including the results of this paper. Each growth rate is represented by a thin line, and thus the seemingly thicker lines indicate accumulation points of growth rates. The area to the right of $\lambda_B$ is completely black because all real numbers at least $\lambda_B$ are known to be growth rates of permutation classes.}
\label{fig-set-of-growth-rates}
\end{figure}

Most recently, in the antecedent paper to this, Vatter~\cite{vatter:growth-rates-of:} established that while there are uncountably many growth rates of permutation classes in every neighborhood of the algebraic integer
\[
	\xi=\mbox{the unique positive root of $x^5-2x^4-x^2-x-1$}\approx 2.30522,
\]
there are only countably many growth rates under $\xi$. In this work we completely determine these growth rates. The known set of growth rates of permutation classes, including those characterized here, is shown in Figure~\ref{fig-set-of-growth-rates}.

While our work builds on \cite{vatter:growth-rates-of:}, we use only one of its results. To state this result we need a few definitions. First, the \emph{(direct) sum} of the permutations $\pi$ of length $k$ and $\sigma$ of length $\ell$ is defined by
\[
	(\pi\oplus\sigma)(i)
	=
	\left\{\begin{array}{ll}
	\pi(i)&\mbox{for $i\in[1,k]$},\\
	\sigma(i-k)+k&\mbox{for $i\in[k+1,k+\ell]$}.
	\end{array}\right.
\]
The sum of $\pi$ and $\sigma$ is shown pictorially on the left of Figure~\ref{fig-sums}. The analogous operation depicted on the right of Figure~\ref{fig-sums} is called the \emph{skew sum} and denoted $\sigma\ominus\pi$. The permutation $\pi$ is said to be \emph{sum indecomposable} if it cannot be written as the sum of two nonempty permutations. Otherwise, $\pi$ is called \emph{sum decomposable} and in this case we can write $\pi$ uniquely as $\alpha_1\oplus\cdots\oplus\alpha_k$ where the $\alpha_i$ are sum indecomposable; in this case the $\alpha_i$ are called the \emph{sum components} of $\pi$.

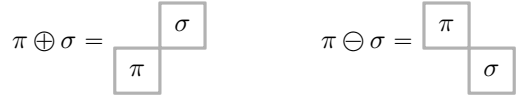
\begin{figure}
\begin{center}
	$\pi\oplus\sigma=$
	\begin{tikzpicture}[baseline=(current bounding box.center), scale=0.2]
		\plotpermbox{0}{0}{2}{2};
		\plotpermbox{3}{3}{5}{5};
		\node at (1,1) {$\pi$};
		\node at (4,4) {$\sigma$};
	\end{tikzpicture}
\quad\quad\quad\quad
	$\pi\ominus\sigma=$
	\begin{tikzpicture}[baseline=(current bounding box.center), scale=0.2]
		\plotpermbox{0}{3}{2}{5};
		\plotpermbox{3}{0}{5}{2};
		\node at (1,4) {$\pi$};
		\node at (4,1) {$\sigma$};
	\end{tikzpicture}
\end{center}
\caption{The sum and skew sum operations.}
\label{fig-sums}
\end{figure}

We say that the class $\C$ is \emph{sum closed} if $\pi\oplus\sigma\in C$ for all $\pi,\sigma\in\C$. Given a permutation class $\C$, its \emph{sum~closure}, denoted by $\bigoplus\C$, is the smallest sum closed permutation class containing $\C$. Equivalently,
\[
	\bigoplus\C
	=
	\{\pi_1\oplus\pi_2\oplus\cdots\oplus\pi_k : \mbox{$\pi_i\in\C$ for all $i$}\}.
\]
It follows from the supermultiplicative version of Fekete's Lemma that every sum closed permutation class has a proper growth rate (a fact first observed by Arratia~\cite{arratia:on-the-stanley-:}). The only result we use from \cite{vatter:growth-rates-of:} is the following.

\begin{theorem}[Vatter~{\cite[Theorem 9.7]{vatter:growth-rates-of:}}]
\label{thm-xi-main}
There are only countably many growth rates of permutation classes below $\xi$ but uncountably many growth rates in every open neighborhood of it. Moreover, every growth rate of a permutation class less than $\xi$ is achieved by a sum closed permutation class.
\end{theorem}

Every permutation class can be specified by a set of permutations that its members avoid, i.e., as
\[
	\Av(B)=\{\pi\st\pi\mbox{ avoids all $\beta\in B$}\}.
\]
Indeed, we may take $B$ to be an antichain (a set of pairwise incomparable elements), and in this case $B$ is unique and is called the \emph{basis} of $\C$. Note that a class is sum closed if and only if all of its basis elements are sum indecomposable. One can also define a permutation class as the \emph{downward closure} of a set $X$ of permutations, i.e., as
\[
	\Sub(X)
	=
	\{\pi\st\pi\le\tau\mbox{ for some $\tau\in X$}\}.
\]
We also make use of the \emph{proper downward closure} of a set $X$ of permutations,
\[
	\PropSub(X)
	=
	\{\pi\st\pi<\tau\mbox{ for some $\tau\in X$}\}.
\]

We define the \emph{generating function} for a class $\C$ of permutations as
\[
	\sum_{\pi\in \C} x^{|\pi|}=\sum_{n\ge 0}|\C_n|x^n,
\]
where $|\pi|$ denotes the length of $\pi$. It is easy to compute generating functions for sum closed classes, assuming we know enough about the sum indecomposable members, as we record below.

\begin{proposition}
\label{prop-enum-oplus-closure}
The generating function for a sum closed permutation class is $1/(1-g)$, where $g$ denotes the generating function for the nonempty sum indecomposable permutations in the class.
\end{proposition}

Growth rates can be determined from generating functions via the following result.

\begin{egf*}[see Flajolet and Sedgewick~{\cite[Section IV.3.2]{flajolet:analytic-combin:}}]
The upper growth rate of a permutation class is equal to the reciprocal of the least positive singularity of its generating function.
\end{egf*}

We say that the sequence $(s_n)$ can be \emph{realized} if there is a permutation class with precisely $s_n$ sum indecomposable permutations for every $n$. In light of Theorem~\ref{thm-xi-main}, our task in this paper is to determine the realizable sequences corresponding to sum closed permutation classes of growth rate less than $\xi$. Our main result is the following.

\begin{thm-xi-grs-main*}
The set of growth rates of permutation classes below $\xi$ can be characterized by a finite set of infinite families of algebraic numbers, collected in Tables~\ref{table-short-good-realizable} and \ref{table-long-good-realizable}.
\end{thm-xi-grs-main*}

We further establish with Theorem~\ref{thm-xi-fb} that every growth rate of a permutation class less than or equal to $\xi$ is achieved by a finitely based class.

In the next section we review the notions of monotone intervals and quotients and state basic consequences of the connection between sum indecomposable permutations and connected graphs. In Section~\ref{sec-xi} we describe several classes with growth rates near $\kappa$ and $\xi$ that reappear later in our arguments. In Section~\ref{sec-reconstruction} we establish that, with two notable exceptions, sum indecomposable permutations are uniquely determined by their sum indecomposable subpermutations. In Section~\ref{sec-taper} we employ this result to establish certain necessary conditions on realizable sequences. These conditions are further refined in Sections~\ref{sec-narrowing-1} and \ref{sec-narrowing-2}. In Section~\ref{sec-realizing} we show that all sequences not eliminated by these considerations are realizable and prove Theorem~\ref{thm-xi-grs-main}. In Section~\ref{sec-fb} we prove Theorem~\ref{thm-xi-fb} and exhibit a counterexample to a suggestion of Klazar~\cite{klazar:overview-of-som:}. We conclude in Section~\ref{sec-conclusion} with a discussion of the obstacles that would have to be overcome to extend our characterization.

\section{Monotone Intervals and Inversion Graphs}
\label{sec-monoints-graphs}

An \emph{interval} in the permutation $\pi$ is a set of contiguous indices $I=[a,b]$ such that the set of values $\pi(I)=\{\pi(i) : i\in I\}$ is also contiguous, and an interval is \emph{nontrivial} if it contains more than one but fewer than all of the entries of $\pi$. The \emph{substitution decomposition} describes how a permutation is built up from a \emph{simple permutation} (one with no nontrivial intervals) via repeated inflations by intervals. (Given a permutation $\sigma$ of length $m$ and nonempty permutations $\alpha_1,\dots,\alpha_m$, the \emph{inflation} of $\sigma$ by $\alpha_1,\dots,\alpha_m$---denoted $\sigma[\alpha_1,\dots,\alpha_m]$---is the permutation of length $|\alpha_1|+\cdots+|\alpha_m|$ obtained by replacing each entry $\sigma(i)$ by an interval that is order isomorphic to $\alpha_i$ in such a way that the intervals themselves are order isomorphic to $\sigma$.)

\begin{figure}
\begin{center}
	\begin{tikzpicture}[scale=0.2]
		% Leftmost 123:
		\draw[lightgray, fill, rotate around={-45:(2,4)}] (2,4) ellipse (25pt and 60pt);
		% Middle 21:
		\draw[lightgray, fill, rotate around={45:(4.5,1.5)}] (4.5,1.5) ellipse (20pt and 40pt);
		% Rightmost 1234:
		\draw[lightgray, fill, rotate around={-45:(7.5,7.5)}] (7.5,7.5) ellipse (30pt and 80pt);
		% The permutation:
		\plotpermbox{0.5}{0.5}{9.5}{9.5};
		\plotperm{3,4,5,2,1,6,7,8,9};
	\end{tikzpicture}
\end{center}
\caption{The plot of $345216789$, whose monotone quotient is $213$.}
\label{fig-479832156}
\end{figure}
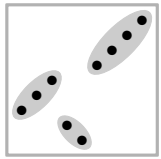

For the results we establish, a different and in some sense weaker decomposition is required. A \emph{monotone interval} in a permutation is an interval whose entries are monotone (increasing or decreasing). Given a permutation $\pi$, we define the \emph{monotone quotient} of $\pi$ to be the shortest permutation $\mu$ such that $\pi$ is an inflation of $\mu$ by monotone permutations. Alternatively, the monotone quotient of $\pi$ can be found by contracting all maximal length monotone intervals to single entries. For example, the monotone quotient of $345216789$ is $213$, because $345216789 = 213[123,21,1234]$. This construction is unique because if two monotone intervals intersect then their union must itself be monotone. Note that the monotone quotient of a sum indecomposable permutation is itself sum indecomposable and that monotone quotients may have nontrivial monotone intervals themselves (unlike the case with the usual substitution decomposition).

Given a permutation $\pi$ and an entry $x$ of $\pi$, we denote by $\pi-x$ the permutation that is order isomorphic to $\pi$ with $x$ removed, and we call $\pi-x$ a \emph{child} of $\pi$. For example, the set of children of $2314$ is
\[
	\{2314-2, 2314-3, 2314-1, 2314-4\}
	=
	\{213, 123, 231\}.
\]
Our interest in monotone intervals and quotients comes from the following fact, which is easily established by using induction on the number of entries between $x$ and $y$.

\begin{proposition}
\label{prop-child-mono}
We have $\pi-x=\pi-y$ for entries $x$ and $y$ of a permutation $\pi$ if and only if $x$ and $y$ lie in the same monotone interval of $\pi$.	
\end{proposition}

As a consequence of Proposition~\ref{prop-child-mono}, if $\pi$ has monotone quotient $\mu$ then by deleting entries of $\pi$ corresponding to distinct entries of $\mu$ we obtain distinct children.

We also make frequent use of the connection between permutations and graphs and a few basic graph-theoretic results. Given a permutation $\pi$ of length $n$, its \emph{inversion graph} $G_\pi$ (a.k.a. a \emph{permutation graph}) is the graph on the vertices $[n]=\{1,2,\dots,n\}$ in which $i\sim j$ if and only if the entries $\pi(i)$ and $\pi(j)$ form an inversion in $\pi$, i.e., if $i<j$ and $\pi(i)>\pi(j)$. Note that the graph $G_\pi$ is connected if and only if $\pi$ is sum indecomposable (because two entries that form an inversion must lie in the same sum component). As each entry of $\pi$ corresponds to a vertex of $G_\pi$, we commit a slight abuse of language by referring (for example) to the degree of an entry of $\pi$ when we mean the degree of the corresponding vertex of $G_\pi$. In a similar way, we talk about the \emph{leaves} of $\pi$ when we mean the entries of $\pi$ that correspond to leaves of $G_\pi$ (i.e., vertices of degree $1$).

The containment order on permutations corresponds to the induced subgraph order on inversion graphs in the sense that if $\sigma\le\pi$ then $G_\pi$ contains an induced subgraph isomorphic to $G_\sigma$. Note that this is true even though the mapping $\pi \to G_\pi$ is not injective; in particular, $G_{\pi}$ and $G_{\pi^{-1}}$ are isomorphic for all permutations $\pi$. We extend our definition of child to graphs by saying that $H$ is a \emph{child} of $G$ if $H=G-v$ for some vertex $v$.

If the inversion graph $G_\pi$ is a path, we call $\pi$ an \emph{increasing oscillation}. This term dates back to Murphy's thesis~\cite{murphy:restricted-perm:}, though note that under our definition, the permutations $1$, $21$, $231$, and $312$ are increasing oscillations while in other works they are not. By direct construction, the increasing oscillations can be seen to be precisely the sum indecomposable permutations that are order isomorphic to subsequences of the \emph{increasing oscillating sequence},
\[
	2,4,1,6,3,8,5,\dots,2k,2k-3,\dots.
\]
Increasing oscillations and their inflations arise repeatedly throughout this work and are depicted in Figures~\ref{fig-infinite-antichain}, \ref{fig-four-antichains}, \ref{fig-mu11-kids}, \ref{fig-5-doesnt-imply-5}, and \ref{fig-forced-extra-kids}.

We denote by $\O_I$ the downward closure of the set of increasing oscillations. There are two increasing oscillations of each length $n\ge 3$, and they are inverses of each other. We arbitrarily choose those beginning with $2$ to call the \emph{primary type}. The two primary type increasing oscillations are
\[
	2 4 1 6 3 8 5 \cdots n (n-3) (n-1)
\]
for even $n\ge 4$, and
\[
	2 4 1 6 3 8 5 \cdots (n-4) n (n-2).
\]
for odd $n\ge 5$.

We conclude this section by stating two graph-theoretic results we appeal to later. Recall that a \emph{cut vertex} in a graph is a vertex whose removal increases the number of connected components.

\begin{proposition}
\label{prop-not-cut}
Every connected graph with at least two vertices contains at least two vertices that are not cut vertices, and the only connected graphs with precisely two non-cut vertices are paths. Correspondingly, every sum indecomposable permutation of length at least two has at least two entries whose removal leaves the permutation sum indecomposable (though the resulting children may be identical).
\end{proposition}

In the following result, we call the complete bipartite graph $K_{1,3}$ a \emph{claw}, and thus the inversion graphs of the permutations $2341$ and $4123$ are claws. Also, the only permutations whose inversion graphs are cycles are $321$ and $3412$.

\begin{proposition}
\label{prop-non-path-child}
Every connected graph that is neither a path, a cycle, nor a claw has a connected child that is not a path. Correspondingly, every sum indecomposable permutation that is not an increasing oscillation, $321$, $2341$, $3412$, or $4123$ contains a sum indecomposable child that is not an increasing oscillation.
\end{proposition}

\section{The Neighborhoods of $\kappa$ and $\xi$}
\label{sec-xi}

Here we briefly survey certain classes that lie at or near the growth rates $\kappa$ and $\xi$. First, because $\O_I$ is sum closed and contains precisely two sum indecomposable permutations of each length $n\ge 3$, we see from Proposition~\ref{prop-enum-oplus-closure} that the generating function of $\O_I$ is
\[
	\frac{1}{1-\left(x+x^2+\frac{2x^3}{1-x}\right)}
	=
	\frac{1-x}{1-2x-x^3}.
\]
We can use the Exponential Growth Formula to compute the growth rate of $\O_I$. In order to express this growth rate as the root of a polynomial (as opposed to the reciprocal of the root of a polynomial), we take the factor responsible for the least positive singularity, $1-2x-x^3$, replace $x$ by $x^{-1}$ to obtain $1-2x^{-1}-x^{-3}$, and then multiply by $x^3$ to get $x^3-2x^2-1$. The growth rate of $\O_I$ is the greatest positive root of this polynomial, namely $\kappa$.

\begin{figure}
\begin{center}
	\begin{tikzpicture}[scale=0.2, baseline=(current bounding box.center)]
		% Intervals:
		\draw[lightgray, fill, rotate around={-45:(1.5,2.5)}] (1.5,2.5) ellipse (20pt and 40pt);
		\draw[lightgray, fill, rotate around={-45:(9.5,10.5)}] (9.5,10.5) ellipse (20pt and 40pt);
		% The permutation:
		\plotpermgraph{2,3,5,1,7,4,9,6,10,11,8};
		\plotpermbox{0.5}{0.5}{11.5}{11.5};
	\end{tikzpicture}
\quad\quad
	\begin{tikzpicture}[scale=0.2, baseline=(current bounding box.center)]
		\node at (0,6.5) {$=$};
	\end{tikzpicture}
\quad\quad
	\begin{tikzpicture}[scale={0.2*(12/16)}, baseline=(current bounding box.center)]
		\absdot{(-2,1)}{};
		\absdot{(2,1)}{};
		\foreach \i in {3,5,7,9,11,13,15}{
			\absdot{(0,\i)}{};
		}
		\absdot{(-2,17)}{};
		\absdot{(2,17)}{};
		\draw (-2,1)--(0,3);
		\draw (2,1)--(0,3);
		\draw (0,3)--(0,15);
		\draw (0,15)--(-2,17);
		\draw (0,15)--(2,17);
	\end{tikzpicture}
\end{center}
\caption{Two drawings of the inversion graph of $\mu_{11}\in U^o$, a double-ended fork.}
\label{fig-infinite-antichain}
\end{figure}
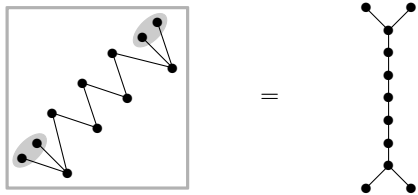

We can use increasing oscillations to build a variety of infinite antichains of permutations whose downward closures also have growth rate $\kappa$. Two such antichains of permutations, one contained in the other, are central to this work; these are denoted $U^o$ and $U$. To construct $U^o$ we take a primary type increasing oscillation (i.e., beginning with $2$) of odd length at least $2i-1\ge 5$ and inflate the two leaves by copies of $12$ (these pairs of entries are called \emph{anchors}). This produces a permutation $\mu_{2i+1}$ of length $2i+1\ge 7$ (whether $U^o$ contains $23451$ differs from paper to paper; we exclude it throughout this work). An example is shown in Figure~\ref{fig-infinite-antichain}.

The inversion graph of $\mu_{2i+1}\in U^o$ is therefore a \emph{double-ended fork}, meaning that it is formed from a path on $2i-3$ vertices by adding four leaves, two adjacent to each endpoint of the original path. The set of double-ended forks clearly forms an infinite antichain under the induced subgraph order. Because $G_\sigma$ is an induced subgraph of $G_\pi$ whenever $\sigma\le\pi$, it follows that the set $U^o$ also forms an infinite antichain of permutations.

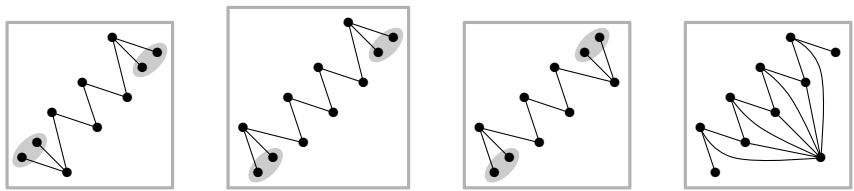
\begin{figure}
\begin{footnotesize}
\begin{center}
	\begin{tikzpicture}[scale=0.2]
		\draw[lightgray, fill, rotate around={-45:(1.5,2.5)}] (1.5,2.5) ellipse (20pt and 40pt);
		\draw[lightgray, fill, rotate around={-45:(9.5,8.5)}] (9.5,8.5) ellipse (20pt and 40pt);
		\plotpermgraph{2,3,5,1,7,4,10,6,8,9};
		\plotpermbox{0.5}{0.5}{10.5}{10.5};
	\end{tikzpicture}
\quad\quad
	\begin{tikzpicture}[scale=0.2]
		\draw[lightgray, fill, rotate around={-45:(2.5,1.5)}] (2.5,1.5) ellipse (20pt and 40pt);
		\draw[lightgray, fill, rotate around={-45:(10.5,9.5)}] (10.5,9.5) ellipse (20pt and 40pt);
		\plotpermgraph{4,1,2,6,3,8,5,11,7,9,10};
		\plotpermbox{0.5}{0.5}{11.5}{11.5};
	\end{tikzpicture}
\quad\quad
	\begin{tikzpicture}[scale=0.2]
		\draw[lightgray, fill, rotate around={-45:(2.5,1.5)}] (2.5,1.5) ellipse (20pt and 40pt);
		\draw[lightgray, fill, rotate around={-45:(8.5,9.5)}] (8.5,9.5) ellipse (20pt and 40pt);
		\plotpermgraph{4,1,2,6,3,8,5,9,10,7};
		\plotpermbox{0.5}{0.5}{10.5}{10.5};
	\end{tikzpicture}
\quad\quad
	\begin{tikzpicture}[scale=0.2]
		\plotperm{4,1,6,3,8,5,10,7,2,9};
		\plotpermbox{0.5}{0.5}{10.5}{10.5};
		% Note: for this one we want curved edges, so we don't want to use \plotpermgraph.
		\draw (1,4)--(2,1); % initial edge
		\draw (1,4)--(4,3)--(3,6)--(6,5)--(5,8)--(8,7)--(7,10)--(10,9);
		\draw plot [smooth, tension=0.75] coordinates { (1,4) (3.25,2) (9,2) };
		\draw plot [smooth, tension=0.75] coordinates { (3,6) (5,4) (9,2) };
		\draw (4,3)--(9,2); % tied by one edges
		\draw plot [smooth, tension=0.75] coordinates { (5,8) (7,6) (9,2) };
		\draw (6,5)--(9,2); % tied by one edges
		\draw plot [smooth, tension=0.75] coordinates { (7,10) (9,7.75) (9,2) };
		\draw (8,7)--(9,2); % tied by one edges
	\end{tikzpicture}
\end{center}
\end{footnotesize}
\caption{From left to right, typical members of four closely related antichains: $U^e$, $(U^o)^{-1}$, $(U^e)^{-1}$, and a typical member of a different antichain based on increasing oscillations.}
\label{fig-four-antichains}
\end{figure}

We further denote by $U$ the set of all permutations whose inversion graphs are isomorphic to double-ended forks on $6$ or more vertices. It follows that $U$ contains $U^o$ together with three other types of members: inflated increasing oscillations of primary type and even length, $U^e$, and the inverses of both $U^o$ and $U^e$. Members of these three additional sets are shown in Figure~\ref{fig-four-antichains}. That figure also includes (on the right) a member of a different antichain based on increasing oscillations. Indeed, there are a great many ways to form infinite antichains from increasing oscillations; the reason that $U^o$ is central to this work while the others are not is essentially because the growth rate of $\bigoplus\Sub(U^o)$ is small.

\begin{figure}
\begin{center}
	\begin{tikzpicture}[scale=0.2]
		\draw[lightgray, fill, rotate around={-45:(1.5,2.5)}] (1.5,2.5) ellipse (20pt and 40pt);
		\plotpermgraph{2, 3, 5, 1, 7, 4, 9, 6, 10, 8};
		\plotpermbox{0.5}{0.5}{10.5}{10.5};
	\end{tikzpicture}
\quad\quad
	\begin{tikzpicture}[scale=0.2]
		\draw[lightgray, fill, rotate around={-45:(8.5,10.5)}] (8.5,10.5) ellipse (20pt and 40pt);
		\plotpermgraph{3, 5, 2, 7, 4, 9, 6, 10, 11, 8};
		\plotpermbox{0.5}{1.5}{10.5}{11.5};
	\end{tikzpicture}
\quad\quad
	\begin{tikzpicture}[scale=0.2]
		\plotpermgraph{2, 4, 1, 6, 3, 8, 5, 10, 7, 9};
		\plotpermbox{0.5}{0.5}{10.5}{10.5};
	\end{tikzpicture}
\quad\quad
	\begin{tikzpicture}[scale=0.2]
		\plotpermgraph{3, 1, 5, 2, 7, 4, 9, 6, 10, 8};
		\plotpermbox{0.5}{0.5}{10.5}{10.5};
	\end{tikzpicture}
\end{center}
\caption{The four sum indecomposable permutations of length $10$ properly contained in $\mu_{2i+3}\in U^o$ for sufficiently large $i$. From left to right, these are the head, the tail, the primary body, and the non-primary body.}
\label{fig-mu11-kids}
\end{figure}
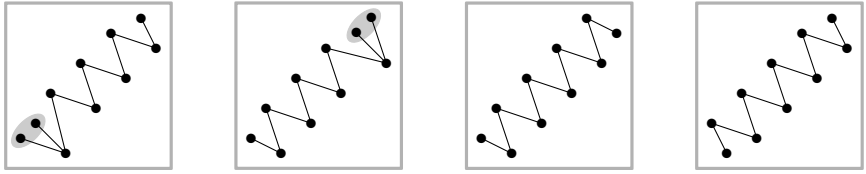

Consider a particular member, say $\mu_{2i+7}$, of $U^o$ for $i\ge 0$. The downward closure of this permutation contains the sum indecomposable permutations $1$, $21$, $231$, $312$, $2341$, $2413$, and $3142$. It also contains $4$ sum indecomposable permutations of length $n$ for each $n$ satisfying $5\le n\le 2i+4$, of the forms shown in Figure~\ref{fig-mu11-kids}. As we increase $n$ beyond $2i+4$, we begin to lose some of these sum indecomposable subpermutations. For $n\ge 2i+5$, $\Sub_n(\mu_{2i+7})$ no longer contains the non-primary body and for $n\ge 2i+6$, $\Sub_n(\mu_{2i+7})$ no longer contains the primary body. A similar analysis applies to members of $U^e$, though they contain $4$ sum indecomposable subpermutations of length $4$ (the additional permutation, $4123$, is contained in the tails of members of $U^e$). We collect these computations below.

\begin{proposition}
\label{prop-U-sum-indecomp-kids}
The sequence counting sum indecomposable permutations contained in $\mu_{2i+7}$ for $i\ge 0$ is $1,1,2,3,4^{2i},3,2,1$. The sequence counting sum indecomposable permutations contained in $\mu_{2i+6}$ for $i\ge 0$ is $1,1,2,4^{2i},3,2,1$.
\end{proposition}

Moreover, it is easy to establish that
\[
	\gr(\Sub(U))=\gr(\Sub(U^o))=\gr(\O_I)=\kappa.
\]
One method is the following. The permutation $\tau$ is said to be a \emph{$p$-point extension} of $\pi$ if $\tau$ can be obtained from $\pi$ by inserting $p$ or fewer entries. The class of all $p$-point extensions of members of the class $\C$ is further denoted $\C^{+p}$. Clearly $|\C_n^{+1}|\le (n+1)^2|\C_n|$, from which it follows that $\ugr(\C^{+p})=\ugr(\C)$ and $\lgr(\C^{+p})=\lgr(\C)$ for all classes $\C$ and natural numbers $p$. The growth rate claim above  holds because $\Sub(U^o)\subseteq\Sub(U)\subseteq\O_I^{+2}$.

Because $U^o$ is an antichain, no member of $U^o$ is contained in any other permutation from $\Sub(U^o)$. Thus every class of the form $\Sub(U^o)\setminus T$ for $T\subseteq U^o$ has growth rate $\kappa$, establishing the following result, which was first obtained by Murphy and Vatter [unpublished] and reported in the conclusion of Klazar~\cite{klazar:on-the-least-ex:}.

\begin{proposition}
\label{prop-uncountable-kappa}
There are uncountably many permutation classes of growth rate $\kappa$.	
\end{proposition}

The results of \cite{vatter:small-permutati:} establish that classes of growth rate less than $\kappa$ cannot contain infinite antichains (i.e., they are \emph{well-quasi-ordered}), and thus Proposition~\ref{prop-uncountable-kappa} is best possible.

The antichain $U^o$ allows us to construct uncountably many permutation classes, but they all have the same growth rate, $\kappa$. In order to construct uncountably many classes with different growth rates, we examine classes lying between $\bigoplus\PropSub(U^o)$ and $\bigoplus\Sub(U^o)$. It follows from our justification of Proposition~\ref{prop-U-sum-indecomp-kids} that the sequence of sum indecomposable permutations in $\PropSub(U^o)$ is $1,1,2,3,4^\infty$. (Note that this is \emph{not} the sequence of sum indecomposable permutations in $\PropSub(U)$ as that class contains two different types of tails of each length as well as an extra sum indecomposable permutation of length $4$.) Consequently, the generating function of $\bigoplus\PropSub(U^o)$ is
\[
	\frac{1}{1-\left(x+x^2+2x^3+3x^4+4\frac{x^5}{1-x}\right)}
	=
	\frac{1-x}{1-2x-x^3-x^4-x^5}.
\]
This shows that the growth rate of $\bigoplus\PropSub(U^o)$ is the greatest positive root of $x^5-2x^4-x^2-x-1$, i.e., $\xi$.

Given infinite sequences or finite sequences padded with zeros, $(r_n)$ and $(t_n)$, we say that $(r_n)$ is \emph{dominated} by $(t_n)$, and write $(r_n)\preceq (t_n)$, if $r_n\le t_n$ for all $n$. We observed above that the sequence counting sum indecomposable permutations in $\PropSub(U^o)$ is $1,1,2,3,4^\infty$. Because $U^o$ itself contains one sum indecomposable permutation of each odd length at least $7$, it follows that the sequence counting sum indecomposable permutations in $\Sub(U^o)$ is $1,1,2,3,4,4,5,4,5,4,\dots$. We can therefore construct sum closed subclasses of $\bigoplus\Sub(U^o)$ with $s_n$ sum indecomposable permutations of length $n$ for every sequence $(s_n)$ satisfying
\[
	(1,1,2,3,4^\infty)\preceq (s_n)\preceq (1,1,2,3,4,4,5,4,5,4,\dots).
\]
Our next result shows that each of these classes has a different growth rate.

\begin{proposition}
\label{prop-xi-uncountable}
For every different sequence $(s_n)$ satisfying
\[
	(1,1,2,3,4^\infty)\preceq (s_n)\preceq (1,1,2,3,4,4,5,4,5,4,\dots),
\]
the growth rate of $1/\left(1-\sum s_nx^n\right)$ is different.
\end{proposition}
\begin{proof}
Let $(s_n)$ and $(t_n)$ be two such sequences and suppose that the least index for which they disagree is $2i+1$ (the cannot disagree at an even index). Without loss of generality we may assume that $s_{2i+1}=4$ while $t_{2i+1}=5$. Define
\[
\begin{array}{rcccccccccccccl}
	g(x)
	&=&
	s_1x&+&s_2x^2&+&\cdots&+&s_{2i}x^{2i}&+&4x^{2i+1}&+&4x^{2i+2}&+&\cdots\\
	&=&
	t_1x&+&t_2x^2&+&\cdots&+&t_{2i}x^{2i}&+&4x^{2i+1}&+&4x^{2i+2}&+&\cdots.
\end{array}
\]
The growth rate of the series expansion of $1/\left(1-\sum s_nx^n\right)$ is therefore at most the growth rate of the series expansion of
\[
	\frac{1}{1-\left(g(x)+\frac{x^{2i+3}}{1-x^2}\right)}
\]
while the growth rate of the series expansion of $1/\left(1-\sum t_nx^n\right)$ is at least the growth rate of the series expansion of
\[
	\frac{1}{1-\left(g(x)+x^{2i+1}\right)}.
\]
We claim that the first quantity is smaller than the second. Because both growth rates are known to be at least $2$, it suffices to consider $0<x<\nicefrac{1}{2}$. For these values of $x$, we have
\[
	\frac{x^{2i+3}}{1-x^2}
	<
	\frac{4}{3}x^{2i+3}
	<
	\frac{1}{3}x^{2i+1}
	<
	x^{2i+1}.
\]
Therefore the smallest real singularity of $1/\left(1-\sum s_nx^n\right)$ is greater than the smallest real singularity of $1/\left(1-\sum t_nx^n\right)$, implying the inequality of growth rates claimed above.
\end{proof}

Proposition~\ref{prop-xi-uncountable} and our previous observations imply the following, first established in~\cite{vatter:permutation-cla:lambda:}.

\begin{proposition}
\label{prop-xi-open-neighborhood}
Every open neighborhood of $\xi\approx 2.30522$ contains uncountably many growth rates of permutation classes.
\end{proposition}

\section{Reconstruction}
\label{sec-reconstruction}

As established by Theorem~\ref{thm-xi-main}, the study of growth rates under $\xi$ is intimately related to the properties of sum indecomposable permutations. Here we establish a general result about sum indecomposable permutations themselves which we apply in the next section to restrict sequences of sum indecomposable permutations in permutation classes.

The permutation $\pi$ is said to be \emph{set reconstructible} if it is the unique permutation with its set of children. Smith~\cite{smith:permutation-rec:} was the first to establish that all permutations of length at least $5$ are set reconstructible (this is the permutation analogue of the graph-theoretic Set Reconstruction Conjecture first stated by Harary~\cite{harary:on-the-reconstr:}), and shortly thereafter Raykova~\cite{raykova:permutation-rec:} gave a simpler proof. Among several examples, the two increasing oscillations of length $4$, $2413$ and $3142$, show that this result is best possible.

\begin{theorem}[Smith~\cite{smith:permutation-rec:}]
\label{thm-reconstruction}
Every permutation of length at least $5$ is uniquely determined by its set of children.
\end{theorem}

We establish an analogous result for sum indecomposable permutations here. To this end, given a sum indecomposable permutation $\pi$ we let $K(\pi)$ denote its set of \emph{sum indecomposable} children,
\[
	K(\pi)=\{\pi-x \st \text{$x$ is an entry of $\pi$ and $\pi-x$ is sum indecomposable}\}.
\]
We also define a closely related set, consisting of the sum indecomposable permutations with at most $m$ sum indecomposable children,
\[
	K^{(m)}
	=
	\{\text{sum indecomposable }\pi \st |K(\pi)|\le m\}.
\]

The main result of this section, below, establishes that $K(\pi)$ uniquely determines $\pi$ except in one notable case.

\begin{theorem}
\label{thm-reconstruction-K}
Every sum indecomposable permutation of length at least $5$ that is not an increasing oscillation is uniquely determined by its set of sum indecomposable children $K(\pi)$.
\end{theorem}

We prove Theorem~\ref{thm-reconstruction-K} via a sequence of propositions. The last of these, Proposition~\ref{prop-no-problems}, represents the majority of the cases and its proof is an adaptation of Raykova's proof of Theorem~\ref{thm-reconstruction}. Before that, we establish the result for the members of $K^{(2)}$, and then for the sum indecomposable permutations $\pi$ of length $n$ such that $\pi-n$ is sum \emph{decomposable}. We begin by characterizing $K^{(1)}$ and $K^{(2)}$.

\begin{proposition}
\label{prop-K1}
For $n\ge 3$, the permutations of length $n$ in $K^{(1)}$ are $n\cdots 21$, $1\ominus (12\cdots (n-1))$, and $(12\cdots (n-1))\ominus 1$.
\end{proposition}
\begin{proof}
Let $\pi\in K^{(1)}$ have monotone quotient $\mu$. Our first goal is to show that $\mu$ must be $1$ or $21$. Indeed, if $|\mu|\ge 3$, then Proposition~\ref{prop-not-cut} shows that the inversion graph $G_\mu$ has at least two non-cut vertices. Label the corresponding entries of $\mu$ as $\mu(i)$ and $\mu(j)$. By deleting entries of $\pi$ corresponding to $\mu(i)$ and $\mu(j)$ we obtain two distinct sum indecomposable children of $\pi$, a contradiction.

If $\mu=1$ then $\pi$ must be decreasing because it is sum indecomposable. In the case where $\mu=21$ we see that at most one of the entries of $\mu$ may be inflated, as otherwise $\pi$ would have more than one sum indecomposable child, completing the proof.
\end{proof}

\begin{proposition}
\label{prop-K2}
The set $K^{(2)}\setminus K^{(1)}$ consists of the sum indecomposable permutations formed by inflating the leaves of an increasing oscillation by monotone intervals.
\end{proposition}
\begin{proof}
Take $\pi \in K^{(2)}$ and let $\mu$ denote the monotone quotient of $\pi$. The inversion graph $G_\mu$ must have at most two non-cut vertices, for if $G_\mu$ had three non-cut vertices then deleting corresponding entries in $\pi$ would give three distinct sum indecomposable children.

Proposition~\ref{prop-not-cut} states that a connected graph has at most two non-cut vertices if and only if it is a path. Thus $\mu$ must be an increasing oscillation. If a non-leaf of $\mu$ is inflated to form $\pi$ then $\pi$ has at least three distinct sum indecomposable children: two from deleting or shrinking a leaf (depending on whether the leaf was inflated), and one from shrinking the inflated non-leaf. Therefore only leaves of $\mu$ may be inflated, and they may only be inflated by monotone intervals because $\mu$ is the monotone quotient of $\pi$. To complete the proof, note that if $\pi$ is formed by inflating the leaves of an increasing oscillation then $\pi$ has at most two sum indecomposable children, each formed by deleting an entry corresponding to a leaf of $\mu$.
\end{proof}

Our next result allows us to reconstruct most members of $K^{(2)}$ from their sets of sum indecomposable children.

\begin{proposition}
\label{prop-K2-recon}
Let $\pi\in K^{(2)}$ have length at least $5$. From $K(\pi)$ one can determine whether $\pi$ is an increasing oscillation. Moreover, if $\pi$ is not an increasing oscillation, it is uniquely determined by $K(\pi)$.
\end{proposition}
\begin{proof}
Obviously, we can determine from $K(\pi)$ whether $\pi\in K^{(1)}$ or $\pi\in K^{(2)}\setminus K^{(1)}$. If $\pi\in K^{(1)}$ then it is one of the three permutations specified by Proposition~\ref{prop-K1} and determining which is trivial. Suppose that $\pi\in K^{(2)}\setminus K^{(1)}$. If both members of $K(\pi)$ are increasing oscillations then $\pi$ must be an increasing oscillation by Proposition~\ref{prop-non-path-child}, but we cannot determine which of the two increasing oscillations it is. Otherwise at least one permutation in $K(\pi)$ is not an increasing oscillation, and thus by Proposition~\ref{prop-K2}, $\pi$ is the result of inflating the leaves of an increasing oscillation by monotone intervals. By examining the corresponding leaves in both children we can determine the monotone quotient of $\pi$, and then the two children together determine $\pi$ uniquely.
\end{proof}

Our next goal is to recognize from $K(\pi)$ for permutations $\pi$ of length $n$ whether $\pi-n$ is sum decomposable, and then to reconstruct those permutations from their sets of sum indecomposable children. In the proofs of these results we make frequent use of the following observation.

\begin{proposition}
\label{prop-delete}
Suppose that $\pi$ is a sum indecomposable permutation of length $n$.
\begin{enumerate}
\item[(a)] If $\pi$ contains at least one entry to the left of $n$ then there is an entry $x$ to the left of $n$ such that $\pi-x$ is sum indecomposable.	
\item[(b)] If $\pi$ contains at least two entries to the right of $n$ then there is an entry $x$ to the right of $n$ such that $\pi-x$ is sum indecomposable.
\end{enumerate}
\end{proposition}
\begin{proof}
Both parts of the proposition follow from the same analysis. Write $\pi$ as the concatenation $\lambda n\rho$. While $\lambda$ is a word over distinct integers instead of a permutation, we can extend the notion of sum components to $\lambda$ in a natural way by defining a sum component of $\lambda$ to be the entries corresponding to a sum component in the permutation that is order isomorphic to $\lambda$. Let $\lambda_1$, $\lambda_2$, $\dots$, $\lambda_m$ be the sum components of $\lambda$.

In both cases, $\rho$ must be nonempty (in case (a) this follows from sum indecomposability and in case (b) it follows by our hypotheses). Let $b$ denote the bottommost entry of $\rho$. Thus we have the picture shown in Figure~\ref{fig-prop-delete}, though note that $\lambda$ may be empty in case (b).

\begin{figure}
\begin{center}
	\begin{tikzpicture}[scale=0.125]
		\plotpermbox{0}{0}{3}{3};
		\plotpermbox{4}{4}{7}{7};
		\node at (1.5,1.5) {$\lambda_1$};
		\node at (5.5,5.5) {$\lambda_2$};
		\node[rotate=45] at (9.6,9.6) {$\dots$};
		\absdot{(12,12)}{};
		\plotpermbox{14}{2}{18}{11};
		\node at (16,6.5) {$\rho$};
		\absdot{(15,1.5)}{};
		\node at (15,1.5) [below right] {$b$};
		\absdot{(2,3.5)}{};
	\end{tikzpicture}
\end{center}	
\caption{The situation in the proof of Proposition~\ref{prop-delete}.}
\label{fig-prop-delete}
\end{figure}
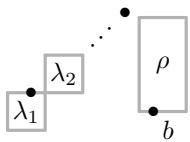

For case (a), because $\pi$ is sum indecomposable, at least one entry of $\lambda_1$ must lie above $b$. If $|\lambda_1|=1$ then every entry of $\lambda$ lies above $b$, and thus we can take $x$ to be the single entry of $\lambda_1$. Otherwise, because every nontrivial connected graph has at least two non-cut vertices (Proposition~\ref{prop-not-cut}), there are at least two entries of $\lambda_1$ whose removal leaves $\lambda_1$ sum indecomposable; let $x$ denote the bottommost of these entries. Removing $x$ leaves at least one entry of $\lambda_1$ above $b$, so $\pi-x$ is sum indecomposable.

Case (b) is trivial if $\lambda$ is empty because then we may remove any entry of $\rho$. Otherwise let $x$ denote any entry of $\rho$ except $b$. In $\pi-x$ there is still an entry of $\lambda_1$ above $b$, so $\pi-x$ is sum indecomposable.
\end{proof}

Given a set $X$ of permutations (always a set of sum indecomposable children in what follows) we define its set of \emph{max locations} by
\[
	\ml(X)
	=
	\{i\st\text{some $\pi\in X$ has its maximum entry at index $i$}\}.
\]
Suppose now that $\pi$ is sum indecomposable, $\pi-n$ is sum decomposable, and $\pi(i)=n$. It follows that $\pi$ has the form shown in Figure~\ref{fig-pi-minus-n-not-SI}, where $\alpha_1$ and $\alpha_2$ are nonempty, $\alpha_1$ is sum indecomposable, and there is an entry in $\alpha_1$ lying to the right of $n$. For all such permutations we have
\[
	\ml(K(\pi)) \subseteq \{i-1, i\}.
\]
Note that this property is recognizable from $K(\pi)$. We now prove that we can also recognize from $K(\pi)$ whether $\pi-n$ is sum decomposable. 

\begin{figure}
\begin{center}
	\begin{tikzpicture}[scale=0.125]
		\draw[dashed] (2.125,11)--(2.125,4.5);
		\plotpermbox{0}{0}{4.25}{4.25};
		\plotpermbox{5.25}{5.25}{9.5}{9.5};
		\plotpartialperm{2.125/11};
		\node at (2,2) {$\alpha_1$};
		\node at (7.25,7.25) {$\alpha_2$};
	\end{tikzpicture}
\end{center}	
\caption{The form of a sum indecomposable permutation $\pi$ with the property that $\pi-n$ is sum decomposable.}
\label{fig-pi-minus-n-not-SI}
\end{figure}
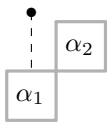

\begin{proposition}
\label{prop-recognize-not-SI}
Let $\pi$ be a sum indecomposable permutation of length at least $5$ that is not an increasing oscillation. It can be determined from $K(\pi)$ whether $\pi-n$ is sum decomposable.
\end{proposition}
\begin{proof}
We make use of the following properties of $K(\pi)$ when $\pi$ is sum indecomposable but $\pi-n$ is sum decomposable, all of which follow readily from the form of such permutations depicted in Figure~\ref{fig-pi-minus-n-not-SI}:
\begin{itemize}
	\item $\ml(K(\pi))\subseteq\{i-1,i\}$ for some $i$;
	\item $K(\pi)$ contains at most one permutation in which the greatest and second greatest entries form a monotone interval;
	\item $K(\pi)$ contains at most one permutation that has its greatest entry in its second-to-last index; and
	\item if $\ml(K(\pi))=\{1,2\}$ then $\pi(2)=n$ and thus $K(\pi)$ would contain at most one child beginning with its maximum entry.
\end{itemize}

Now take $\pi$ to be a sum indecomposable permutation of length $n\ge 5$ for which $\pi-n$ is sum indecomposable. We seek to show that $K(\pi)$ violates one of the conditions above. Thus we may begin by assuming that $\ml(K(\pi))\not\subseteq\{i-1,i\}$ for some $i$. We are done if $\pi\in K^{(2)}$ by Proposition~\ref{prop-K2-recon}, so we may assume that $\pi$ has at least three sum indecomposable children. Throughout the proof we let $k$ denote the position of $n$ in $\pi$ and $j$ the position of $n-1$. We distinguish three cases, depending on whether $|k-j|$ is $1$, $2$, or at least $3$.

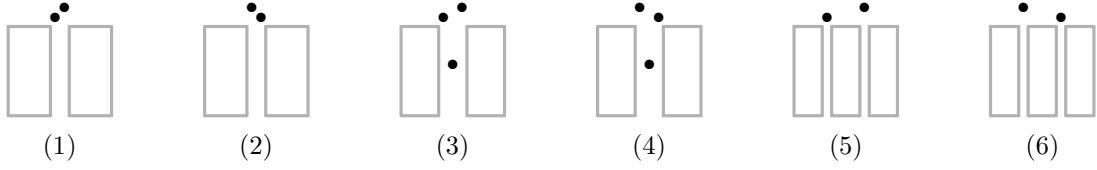
\begin{figure}
\begin{center}
\begin{tabular}{ccccccccccc}
\begin{tikzpicture}[scale=0.125]
		\plotpermbox{0}{0}{3.5}{8.5};
		\plotpermbox{6.5}{0}{10}{8.5};
		\plotpartialperm{4.5/10,5.5/11};
	\end{tikzpicture}
& \quad\quad &
	\begin{tikzpicture}[scale=0.125]
		\plotpermbox{0}{0}{3.5}{8.5};
		\plotpermbox{6.5}{0}{10}{8.5};
		\plotpartialperm{4.5/11,5.5/10};
	\end{tikzpicture}
& \quad\quad &
	\begin{tikzpicture}[scale=0.125]
		\plotpermbox{0}{0}{3}{8.5};
		\plotpermbox{7}{0}{10}{8.5};
		\plotpartialperm{4/10,5/5,6/11};
	\end{tikzpicture}
& \quad\quad &
	\begin{tikzpicture}[scale=0.125]
		\plotpermbox{0}{0}{3}{8.5};
		\plotpermbox{7}{0}{10}{8.5};
		\plotpartialperm{4/11,5/5,6/10};
	\end{tikzpicture}
& \quad\quad &
	\begin{tikzpicture}[scale=0.125]
		\plotpermbox{0}{0}{2}{8.5};
		\plotpermbox{4}{0}{6}{8.5};
		\plotpermbox{8}{0}{10}{8.5};
		\plotpartialperm{3/10,7/11};
	\end{tikzpicture}
& \quad\quad &
	\begin{tikzpicture}[scale=0.125]
		\plotpermbox{0}{0}{2}{8.5};
		\plotpermbox{4}{0}{6}{8.5};
		\plotpermbox{8}{0}{10}{8.5};
		\plotpartialperm{3/11,7/10};
	\end{tikzpicture}
\\[2pt]
	(1)&&(2)&&(3)&&(4)&&(5)&&(6)
\end{tabular}
\end{center}
\caption{The six cases in the proof of Proposition~\ref{prop-recognize-not-SI}.}
\label{fig-ml-cases}
\end{figure}

We first consider the cases where $|k-j|=1$, i.e., where $n$ and $n-1$ form a monotone interval, cases (1) and (2) of Figure~\ref{fig-ml-cases}. Because $\pi$ has at least three sum indecomposable children, it has at least two sum indecomposable children whose greatest and second greatest entries form a monotone interval. However, we have already remarked that this could not occur if $\pi-n$ were sum decomposable.

Next we handle the most difficult of the cases, where $|k-j|=2$, cases (3) and (4) of Figure~\ref{fig-ml-cases}. In case (3), the sum indecomposability of $\pi-n$ implies that $j \in \ml(K(\pi))$. Furthermore, if the rightmost box has at least two entries, then one may be deleted to give a sum decomposable child of $\pi$ that has its maximum entry at index $k$ (by Proposition~\ref{prop-delete} (b)). Since the rightmost box cannot be empty ($\pi$ is sum indecomposable), it follows that this case is complete unless the rightmost box contains exactly one entry. In that case, there are at least two sum indecomposable children that have their maximum entries in the second-to-last index, which could not occur if $\pi-n$ were sum decomposable.

In case (4), because we are assuming that $\pi-n$ is sum indecomposable, we have $j-1=k+1\in\ml(K(\pi))$. If there are any entries to left of $n$ in $\pi$ then Proposition~\ref{prop-delete} (a) shows that one of those entries can be deleted to obtain a sum indecomposable child of $\pi$, implying that $k-1\in\ml(K(\pi))$, and thus we are done with that case. Otherwise $\pi$ begins with its maximum, and it follows that $\ml(K(\pi))=\{1,2\}$. Because $\pi$ has at least three sum indecomposable children, there are at least two permutations in $K(\pi)$ that begin with their maximum entries, which could not occur if $\pi-n$ were sum indecomposable.

This leaves us with the cases where $|k-j|\ge 3$, cases (5) and (6) of Figure~\ref{fig-ml-cases}. In this case, the sum indecomposability of $\pi-n$ shows that either $j$ or $j-1$ lies in $\ml(K(\pi))$. Because $\pi$ has another sum indecomposable child, we see that either $k$ or $k-1$ lies in $\ml(K(\pi))$. Therefore $\ml(K(\pi))$ is not of the form $\{i-1,i\}$, completing the proof.
\end{proof}

Next we reconstruct the permutations recognized by the previous result.

\begin{proposition}
\label{prop-reconstruct-not-SI}
Let $\pi$ be a sum indecomposable permutation of length at least $5$ that is not an increasing oscillation. If $\pi-n$ is sum decomposable then $K(\pi)$ uniquely determines $\pi$.
\end{proposition}
\begin{proof}
By our previous proposition, we can recognize from $K(\pi)$ whether $\pi-n$ is sum decomposable. Assuming that it is sum decomposable, it has the form shown in Figure~\ref{fig-pi-minus-n-not-SI}, where $\alpha_1$ is nonempty and sum indecomposable and $\alpha_2$ is also nonempty.

As $\pi-x$ is sum indecomposable for all entries $x$ in $\alpha_2$ in this case, we can determine the subpermutation contained in $\alpha_1$ by looking for the longest initial sum indecomposable subpermutation (ignoring the maximum entry) in any member of $K(\pi)$. From this we also identify the location of the entry $n$ in $\pi$. We can then determine the contiguous set of indices and values that make up the entries in $\alpha_2$. As we know there is some entry $x$ in $\alpha_1$ such that $\pi-x$ is sum indecomposable (Proposition~\ref{prop-delete} (a)), there is a permutation in $K(\pi)$ that contains $\alpha_2$ intact. After identifying this child, we have all the information necessary to reconstruct $\pi$ from $K(\pi)$.
\end{proof}

Because the inverse of a sum indecomposable permutation is also sum indecomposable, Proposition~\ref{prop-recognize-not-SI} shows that we can also recognize from $K(\pi)$ whether $\pi-\pi(n)$ is sum indecomposable, and Proposition~\ref{prop-reconstruct-not-SI} shows that we can reconstruct $\pi$ from $K(\pi)$ whenever $\pi-\pi(n)$ is sum decomposable. It remains only to reconstruct those permutations $\pi$ for which both $\pi-n$ and $\pi-\pi(n)$ are sum indecomposable, which we do in the final result of this section to complete the proof of Theorem~\ref{thm-reconstruction-K}. The proof we give is an adaptation of that given in the general reconstruction context by Raykova~\cite{raykova:permutation-rec:}.

\begin{proposition}
\label{prop-no-problems}
If $\pi$ is a sum indecomposable permutation of length at least $5$ and $\pi-n$ and $\pi-\pi(n)$ are both sum indecomposable then $K(\pi)$ uniquely determines $\pi$.
\end{proposition}
\begin{proof}
We prove the claim by induction on $n$, the base case of $n=5$ being readily verified by computer.

Let $\pi$ be as in the statement of the proposition and suppose that $\tau$ is a sum indecomposable permutation such that $K(\tau)=K(\pi)$. We show that this implies $\pi=\tau$. Proposition~\ref{prop-recognize-not-SI} shows that $\tau-n$ must be sum indecomposable. Additionally, by applying Proposition~\ref{prop-recognize-not-SI} to the set of inverses of $K(\tau)=K(\pi)$ we see that $\tau-\tau(n)$ must also be sum indecomposable. If both $\pi-n=\tau-n$ and $\pi-\pi(n)=\tau-\tau(n)$ then it follows that either $\pi=\tau$ or $\pi$ or $\tau$ ends in $n$. As $\pi$ and $\tau$ are assumed to be sum indecomposable, neither can end in $n$. Therefore in this case, $\pi = \tau$.

Thus we can assume that either $\pi-n \neq \tau-n$ or $\pi-\pi(n) \neq \tau - \tau(n)$. By symmetry, we may assume the former, and therefore by induction, we have that $K(\pi-n)\neq K(\tau-n)$. Therefore there is some permutation in one set but not the other; suppose without loss of generality that $\sigma\in K(\pi-n)\setminus K(\tau-n)$, so there is an entry $x$ of $\pi$ such that $\pi-n-x=\sigma$. We claim that $\pi-x\in K(\pi)\setminus K(\tau)$, a contradiction which will complete the proof.

To prove this claim we must first show that $\pi-x\in K(\pi)$. Suppose otherwise, so $\pi$, $\pi-n$, and $\pi-n-x$ are all sum indecomposable, but $\pi-x$ is sum decomposable. The only way this can happen is if $n$ is the second to last entry of $\pi$, which violates the assumption that $\pi-\pi(n)$ is sum indecomposable. Hence, $\pi-x \in K(\pi)$.

Lastly we show that $\pi-x\notin K(\tau)$. Suppose to the contrary that $\pi-x \in K(\tau)$. Then there exists an entry $y$ in $\tau$ such that $\pi - x = \tau - y$. If $y=n$, then $\pi - n -x = \tau - n - (n-1) \in K(\tau-n)$, a contradiction. If $y\neq n$, then $\pi-n-x = \tau-n-y \in K(\tau-n)$, again a contradiction. Therefore, since $\pi-x \in K(\pi)\setminus K(\tau)$, we have derived a contradiction. It must be the case that $\pi = \tau$.
\end{proof}

\section{Hereditary Sets of Sum Indecomposable Permutations}
\label{sec-taper}

The purpose of this section is to establish results of the form
\begin{quote}
If the class $\C$ contains $m$ sum indecomposable permutations of length $n$ then $\C$ contains at least $m$ sum indecomposable permutations of length $n-1$.
\end{quote}
One such result is obvious---because every connected graph contains a connected induced subgraph on one fewer vertex, if a permutation class contains a sum indecomposable permutation of length $n\ge 2$ then it also contains a sum indecomposable permutation of length $n-1$.

\begin{figure}
\begin{center}
	\begin{tikzpicture}[scale=0.2]
		% Intervals:
		\draw[lightgray, fill, rotate around={-45:(1.5,2.5)}] (1.5,2.5) ellipse (20pt and 40pt);
		\draw[lightgray, fill, rotate around={-45:(9.5,10.5)}] (9.5,10.5) ellipse (20pt and 40pt);
		% The permutation:
		\plotpermgraph{2,3,5,1,7,4,9,6,10,11,8};
		\plotpermbox{0.5}{0.5}{11.5}{11.5};
	\end{tikzpicture}
\quad
	\begin{tikzpicture}[scale=0.2]
		% Intervals:
		\draw[lightgray, fill, rotate around={-45:(1.5,2.5)}] (1.5,2.5) ellipse (20pt and 40pt);
		% The permutation:
		\plotpermgraph{2,3,5,1,7,4,9,6,11,8,10};
		\plotpermbox{0.5}{0.5}{11.5}{11.5};
	\end{tikzpicture}
\quad
	\begin{tikzpicture}[scale=0.2]
		% Intervals:
		\draw[lightgray, fill, rotate around={-45:(9.5,10.5)}] (9.5,10.5) ellipse (20pt and 40pt);
		% The permutation:
		\plotpermgraph{3,1,5,2,7,4,9,6,10,11,8};
		\plotpermbox{0.5}{0.5}{11.5}{11.5};
	\end{tikzpicture}
\quad
	\begin{tikzpicture}[scale=0.2]
		\plotpermgraph{3,1,5,2,7,4,9,6,11,8,10};
		\plotpermbox{0.5}{0.5}{11.5}{11.5};
	\end{tikzpicture}
\quad
	\begin{tikzpicture}[scale=0.2]
		\plotpermgraph{2, 4, 1, 6, 3, 8, 5, 10, 7, 11, 9};
		\plotpermbox{0.5}{0.5}{11.5}{11.5};
	\end{tikzpicture}
\\[12pt]
	\begin{tikzpicture}[scale=0.2]
		% Intervals:
		\draw[lightgray, fill, rotate around={-45:(1.5,2.5)}] (1.5,2.5) ellipse (20pt and 40pt);
		% The permutation:
		\plotpermgraph{2,3,5,1,7,4,9,6,10,8};
		\plotpermbox{0.5}{0.5}{10.5}{10.5};
	\end{tikzpicture}
\quad
	\begin{tikzpicture}[scale=0.2]
		% Intervals:
		\draw[lightgray, fill, rotate around={-45:(8.5,10.5)}] (8.5,10.5) ellipse (20pt and 40pt);
		% The permutation:
		\plotpermgraph{3,5,2,7,4,9,6,10,11,8};
		\plotpermbox{0.5}{1.5}{10.5}{11.5};
	\end{tikzpicture}
\quad
	\begin{tikzpicture}[scale=0.2]
		\plotpermgraph{3,1,5,2,7,4,9,6,10,8};
		\plotpermbox{0.5}{0.5}{10.5}{10.5};
	\end{tikzpicture}
\quad
	\begin{tikzpicture}[scale=0.2]
		\plotpermgraph{2, 4, 1, 6, 3, 8, 5, 10, 7, 9};
		\plotpermbox{0.5}{0.5}{10.5}{10.5};
	\end{tikzpicture}
\end{center}
\caption{The $5$ sum indecomposable permutations of length $11$ shown on the top row (the first of which is a member of $U^o$) together contain only $4$ sum indecomposable permutations of length $10$, shown on the bottom row (which is a repeat of Figure~\ref{fig-mu11-kids}).}
\label{fig-5-doesnt-imply-5}
\end{figure}
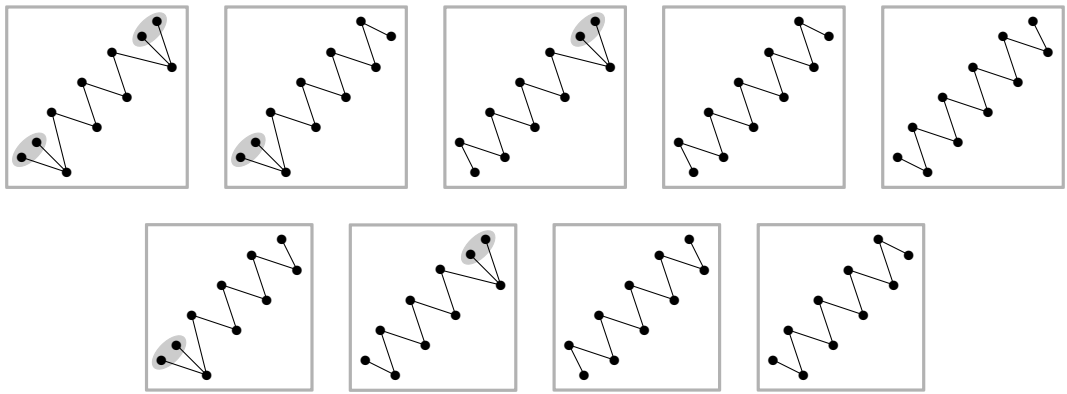

In general, however, we must be careful with our choices of $m$ and $n$. For a trivial example, there are $3$ sum indecomposable permutations of length $3$, but there is only $1$ of length $2$. Thus our results must impose restrictions on $n$. More significantly, no such result holds for $m=5$, with any restrictions on $n$, as witnessed by the set of permutations depicted in Figure~\ref{fig-5-doesnt-imply-5}. Note that the permutations from this example are contained in $\Sub(U^o)$. 

The main results of this section are collected below.
\begin{itemize}
\item {\bf Proposition~\ref{prop-2-implies-2}.} \emph{If the permutation class $\C$ contains $2$ sum indecomposable permutations of length $n \geq 4$ then it contains at least $2$ sum indecomposable permutations of length $n-1$.}
\item {\bf Proposition~\ref{prop-3-implies-3}.} \emph{If the permutation class $\C$ contains $3$ sum indecomposable permutations of length $n \geq 5$ then it contains at least $3$ sum indecomposable permutations of length $n-1$.}
\item {\bf Proposition~\ref{prop-4-implies-4}.} \emph{If the permutation class $\C$ contains $4$ sum indecomposable permutations of length $n \geq 6$ then it contains at least $4$ sum indecomposable permutations of length $n-1$.}
\end{itemize}
The first of these results was proved by Kaiser and Klazar~\cite{kaiser:on-growth-rates:}, though we present a shorter proof via reconstruction. The second of these results was claimed in \cite[Proposition A.16]{vatter:small-permutati:}, but the argument presented there is faulty and our proof corrects this error. Proposition~\ref{prop-4-implies-4} is original, and best possible in light of the example depicted in Figure~\ref{fig-5-doesnt-imply-5}.

\begin{proposition}
\label{prop-2-implies-2}
If the permutation class $\C$ contains $2$ sum indecomposable permutations of length $n \geq 4$ then it contains at least $2$ sum indecomposable permutations of length $n-1$.
\end{proposition}
\begin{proof}
Suppose that $\C$ contains two sum indecomposable permutations of length $n\ge 4$, say $\pi_1$ and $\pi_2$. As the result is easily checked for $n=4$, we may assume that $n\ge 5$. If either $\pi_1$ or $\pi_2$ has two sum indecomposable children then we are immediately done. Thus we may assume that they each have precisely one sum indecomposable child. Theorem~\ref{thm-reconstruction-K} (which guarantees reconstruction from sum indecomposable children) implies that these children must be different, completing the proof.
\end{proof}

As a consequence of our characterization of $K^{(1)}$ in Proposition~\ref{prop-K1}, we obtain the proposition below, which we use in what follows.

\begin{proposition}
\label{prop-no-2K1}
Every sum indecomposable permutation of length $n\ge 5$ that does not lie in $K^{(1)}$ has at least one sum indecomposable child that does not lie in $K^{(1)}$.
\end{proposition}
\begin{proof}
Suppose that $\pi\notin K^{(1)}$ has length at least $5$ and let $\sigma_1$ and $\sigma_2$ denote two sum indecomposable children of $\pi$. If $\sigma_1,\sigma_2\in K^{(1)}$ then by Proposition~\ref{prop-K1} we must, up to relabeling, have $\sigma_1=1\ominus(12\cdots(n-2))$ and $\sigma_2=(12\cdots(n-2))\ominus 1$, because $\pi$ cannot contain one of these permutations together with $(n-1)\cdots 21$. This however implies that $\pi=1\ominus(12\cdots(n-2))\ominus 1$, which has $1\ominus(12\cdots(n-3))\ominus 1$ as a sum indecomposable child, and as $n\ge 5$, this permutation does not lie in $K^{(1)}$---in fact, it lies in $K^{(3)}\setminus K^{(2)}$.
\end{proof}

We now prove our second result.

\begin{proposition}
\label{prop-3-implies-3}
If the permutation class $\C$ contains $3$ sum indecomposable permutations of length $n \geq 5$ then it contains at least $3$ sum indecomposable permutations of length $n-1$. 
\end{proposition}
\begin{proof}
Suppose that $\C$ contains three sum indecomposable permutations of length $n\ge 5$, say $\pi_1$, $\pi_2$, and $\pi_3$, labeled so that $|K(\pi_1)|\ge |K(\pi_2)|\ge |K(\pi_3)|$. We may assume that $\pi_1,\pi_2,\pi_3\in K^{(2)}$ because otherwise we are done. If $\pi_1$ (and consequently also $\pi_2$ and $\pi_3$) has only a single sum indecomposable child, then reconstruction (Theorem~\ref{thm-reconstruction-K}) shows that these three children must be distinct, and we are done.

Suppose now that $\pi_1$ has two sum indecomposable children. If $\pi_2$ also has two sum indecomposable children then we are done by reconstruction unless both $\pi_1$ and $\pi_2$ are increasing oscillations. In this case, $\pi_3$ cannot be an increasing oscillation (there are only two of each length), and so by reconstruction it must have a sum indecomposable child that is not a child of $\pi_1$ or $\pi_2$.

If $\pi_1$ has two sum indecomposable children and $\pi_2$ and $\pi_3$ have one sum indecomposable child each, we see from Proposition~\ref{prop-no-2K1} that $\pi_1$ has a sum indecomposable child that does not lie in $K^{(1)}$ and therefore differs from the sum indecomposable children of $\pi_2$ and $\pi_3$, which must also differ from each other by reconstruction.
\end{proof}

From our characterizations of $K^{(1)}$ and $K^{(2)}$ in Propositions~\ref{prop-K1} and \ref{prop-K2}, we obtain the following analogue of Proposition~\ref{prop-no-2K1}.

\begin{proposition}
\label{prop-no-3K2}
Every sum indecomposable permutation of length $n\ge 7$ that does not lie in $K^{(2)}$ has at least one sum indecomposable child that does not lie in $K^{(2)}$.\footnote{The restriction on length is needed because the permutation $324651\in K^{(3)}\setminus K^{(2)}$ has three sum indecomposable children, $23541$, $32451$, and $32541$, all of which belong to $K^{(2)}$.}
\end{proposition}
\begin{proof}
Suppose that $\pi\notin K^{(2)}$ has length $n\ge 7$ and let $\mu$ denote its monotone quotient. If $\mu$ is an increasing oscillation then $\pi$ must contain a non-trivial monotone interval corresponding to a non-leaf of $\mu$. Since $n\ge 7$, either $|\mu|\ge 4$ or some entry of $\mu$ is inflated by an interval of length at least $3$. In either case, an entry of $\pi$ can be deleted to yield a child of $\pi$ not in $K^{(2)}$.

In the case where $\mu$ is not a path, we consider two subcases. If $|\mu|\ge 5$, then by Proposition~\ref{prop-non-path-child}, $\mu$ contains a sum indecomposable child that is not an increasing oscillation, from which it follows immediately that $\pi$ contains a sum indecomposable child that does not lie in $K^{(2)}$. Otherwise $|\mu|\le 4$. If $\pi$ has four non-trivial monotone intervals, we see from the characterization of $K^{(2)}$ that $\pi$ has no sum indecomposable children in $K^{(2)}$ and thus satisfies the conclusion of the proposition. Otherwise, because $n\ge 7$ and $|\mu|\le 4$, $\pi$ contains a monotone interval of length at least $3$. Deleting an entry from this interval yields a sum indecomposable child of $\pi$ that does not lie in $K^{(2)}$.
\end{proof}

We now establish the final result of the section.

\begin{proposition}
\label{prop-4-implies-4}
	If the permutation class $\C$ contains $4$ sum indecomposable permutations of length $n \geq 6$ then it contains at least $4$ sum indecomposable permutations of length $n-1$. 
\end{proposition}
\begin{proof}
The case $n=6$ can be checked by computer. Suppose that $\C$ contains sum indecomposable permutations $\pi_1$, $\pi_2$, $\pi_3$, and $\pi_4$ of length $n \geq 7$, labeled so that $|K(\pi_1)| \geq |K(\pi_2)| \geq |K(\pi_3)| \geq |K(\pi_4)|$. We may assume that $\pi_1\in K^{(3)}$, as otherwise we are done. We use the observation that the sets $K^{(1)}$ and $K^{(2)}$ are closed downward\footnote{It can be shown that all sets $K^{(m)}$ are closed downward in this sense, but we only use this fact for $K^{(1)}$ and $K^{(2)}$.} in the sense that if $\pi$ and $\sigma$ are sum indecomposable and $\sigma \leq \pi$, then $\pi \in K^{(1)}$ implies $\sigma \in K^{(1)}$, and $\pi \in K^{(2)}$ implies $\sigma \in K^{(2)}$. This is easily seen by examining the forms of such permutations established by Propositions~\ref{prop-K1} and \ref{prop-K2}.

The proof is via case analysis. First we treat the case $|K(\pi_1)| = 3$. If $|K(\pi_2)| = 3$ as well, then reconstruction (Theorem~\ref{thm-reconstruction-K}) guarantees that $\pi_2$ has a child different from those of $\pi_1$, and the proof is complete. Otherwise let $K(\pi_1) = \{\sigma_1, \sigma_2, \sigma_3\}$. By Proposition~\ref{prop-no-3K2} we may assume that $\sigma_1\notin K^{(2)}$. Proposition~\ref{prop-3-implies-3} then shows that $\pi_2$, $\pi_3$, and $\pi_4$ together have at least three sum indecomposable children, none equal to $\sigma_1$ because $\sigma_1\notin K^{(2)}$, completing this case.

We proceed to the case where $|K(\pi_1)| = 2$, which is more delicate than the previous case. Suppose first that $|K(\pi_2)| = 2$ and that $\pi_1$ and $\pi_2$ are not both increasing oscillations. Then together $\pi_1$ and $\pi_2$ have at least three sum indecomposable children, say $\sigma_1$, $\sigma_2$, and $\sigma_3$, by reconstruction.

Suppose that $|K(\pi_3)|=2$. If $K(\pi_3)\subseteq\{\sigma_1,\sigma_2,\sigma_3\}$ then by Proposition~\ref{prop-no-2K1} and reconstruction, at least two of these children do not lie in $K^{(1)}$; suppose that $\sigma_1,\sigma_2\notin K^{(1)}$. If $|K(\pi_4)|=2$ then we are done by reconstruction as four permutations from $K^{(2)}$ cannot together have only $3$ sum indecomposable children. Thus we may assume that $\pi_4\in K^{(1)}$ and so we are done unless $K(\pi_4)=\{\sigma_3\}$. However, simple case analysis based on the three possibilities for $\sigma_3$ shows that this is not possible.

If instead $\pi_3,\pi_4\in K^{(1)}$ then (up to labeling) the only way the $\pi_i$ could share $3$ sum indecomposable children is as in the diagram below.
	
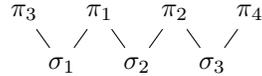
\[
		\begin{tikzpicture}[baseline=(current bounding box.center)]
			\node (A) at (0,0) {$\pi_3$};
			\node (B) at (1,0) {$\pi_1$};
			\node (C) at (2,0) {$\pi_2$};
			\node (D) at (3,0) {$\pi_4$};
			\node (E) at (0.5, -0.7) {$\sigma_1$};
			\node (F) at (1.5, -0.7) {$\sigma_2$};
			\node (G) at (2.5, -0.7) {$\sigma_3$};
			\draw (A) -- (E);
			\draw (B) -- (E);
			\draw (B) -- (F);
			\draw (C) -- (F);
			\draw (C) -- (G);
			\draw (D) -- (G);
		\end{tikzpicture}
	\]
In this case we may assume by symmetry that $\sigma_1=1\ominus (12\cdots(n-2))$. Moreover, $\sigma_3$ is obtained from $\sigma_1$ by inserting two entries and deleting two entries. Thus $\sigma_3$ cannot be decreasing, so the only possibility is that $\sigma_3=(12\cdots(n-2))\ominus 1$. In this case it is straight-forward to see that no choice of $\pi_1$ and $\pi_2$ makes this diagram possible.

If $\pi_1$ and $\pi_2$ are increasing oscillations, we may assume that $|K(\pi_3)| = |K(\pi_4)| = 1$ as otherwise we are in a case that we have already treated. By reconstruction, the sum indecomposable children of $\pi_3$ and $\pi_4$ must be different, and neither of the two can be increasing oscillations because they lie in $K^{(1)}$, so $\pi_1$, $\pi_2$, $\pi_3$, and $\pi_4$ together have $4$ sum indecomposable children, as desired.
	
There are two final cases. If $|K(\pi_1)| = 2$ and $|K(\pi_2)| = |K(\pi_3)| = |K(\pi_4)| = 1$ then $\pi_2$, $\pi_3$, and $\pi_4$ together have $3$ sum indecomposable children by reconstruction, all of which belong to $K^{(1)}$, while $\pi_1$ has a sum indecomposable child that does not lie in $K^{(1)}$, giving the $4$ sum indecomposable children we need. Finally, we cannot have $|K(\pi_i)|=1$ for all $i$ because there are precisely $3$ such permutations of each length. 
\end{proof}

\section{Narrowing the Search I --- Basic Calculations}
\label{sec-narrowing-1}

Recall that we say the sequence $(s_n)$ can be realized if there is a permutation class with $s_n$ sum indecomposable permutations of every length $n\ge 1$. We now begin the investigation of sequences of sum indecomposable permutations that lead to classes with growth rates less than $\xi$ and of those, which can be realized. We separate this endeavor into roughly three tasks. In this section, we establish basic results on those sequences of sum indecomposable permutations leading to growth rates greater than $\xi$. In the next section, we show that certain sequences not eliminated in this section still cannot be realized. Finally, in Section~\ref{sec-realizing}, we realize the remaining sequences.

The results of the previous section impose restrictions on how a sequence $(s_n)$ of sum indecomposable permutations can taper. In particular, we know that (for $n$ large enough) once the sequence drops below $4$, it stays below $4$, once it drops below $3$, it stays below $3$, once it drops below $2$, it stays below $2$, and once it hits $0$, it stays at $0$. More precisely:
\begin{itemize}
\item for $n\ge 1$, if $s_n=0$ then $s_{n+1}=0$ (trivially),
\item for $n\ge 3$, if $s_n\le 1$ then $s_{n+1}\le 1$ (by Proposition~\ref{prop-2-implies-2}),
\item for $n\ge 4$, if $s_n\le 2$ then $s_{n+1}\le 2$ (by Proposition~\ref{prop-3-implies-3}), and
\item for $n\ge 5$, if $s_n\le 3$ then $s_{n+1}\le 3$ (by Proposition~\ref{prop-4-implies-4}).
\end{itemize}
We say that the sequence $(s_n)$ is \emph{legal} if it satisfies these five rules as well as $s_1\le 1$, $s_2\le 1$, and $s_3\le 3$ (these are the numbers of sum indecomposable permutations of those lengths). Only legal sequences of sum indecomposable permutations can be realized, though as we will see, not all legal sequences can be realized.

Recall from Section~\ref{sec-xi} that we say that the sequence $(r_n)$ is dominated by $(t_n)$, written $(r_n)\preceq (t_n)$, if $r_n\le t_n$ for all $n$. If the sequence $(r_n)$ of sum indecomposable permutations leads to a growth rate greater than $\xi$, every sequence dominating $(r_n)$ also leads to a growth rate greater than $\xi$.

\begin{table}
\begin{footnotesize}
\[
	\begin{array}{lll}
	\hline
	\mbox{sequence}&\mbox{growth rate is the greatest real root of}&\mbox{bound}
	\\\hline
	\\[-8pt]
		1,1,2,4,3,3,2,1
		&
		x^5-2x^4-x^2-x-1
		&
		\xi\approx 2.30522
	\\[1pt]
		1,1,2,4,3,3,3
		&
		x^7-x^6-x^5-2x^4-4x^3-3x^2-3x-3
		&
		\phantom{\xi}> 2.30688
	\\[1pt]
		1,1,2,4,4,1,1,1,1,1,1
		&
		x^{11}-x^{10}-x^9-2x^8-4x^7-4x^6-x^5-x^4-x^3-x^2-x-1
		&
		\phantom{\xi}> 2.30525
	\\[1pt]
		1,1,2,4,4,2
		&
		x^6-x^5-x^4-2x^3-4x^2-4x-2
		&
		\phantom{\xi}> 2.30692
	\\[1pt]
		1,1,2,4,5
		&
		x^5-x^4-x^3-2x^2-4x-5
		&
		\phantom{\xi}> 2.30902
	\\[1pt]
		1,1,2,5,2,1,1
		&
		x^6-2x^5+x^4-3x^3-2x^2-1
		&
		\phantom{\xi}> 2.30790
	\\[1pt]
		1,1,2,5,2,2
		&
		x^6-x^5-x^4-2x^3-5x^2-2x-2
		&
		\phantom{\xi}> 2.31179
	\\[1pt]
		1,1,2,5,3
		&
		x^5-x^4-x^3-2x^2-5x-3
		&
		\phantom{\xi}> 2.31392
%	\\[1pt]
%		1,1,2,6,1
%		&
%		x^5-x^4-x^3-2x^2-6x-1
%		&
%		\phantom{\xi}> 2.31891
%	\\[1pt]
%		1,1,2,7
%		&
%		x^4-x^3-x^2-2x-7
%		&
%		\phantom{\xi}> 2.33951
	\\[1pt]
		1,1,3,3,1,1,1,1,1,1
		&
		x^{10}-x^9-x^8-3x^7-3x^6-x^5-x^4-x^3-x^2-x-1
		&
		\phantom{\xi}> 2.30528
	\\[1pt]
		1,1,3,3,2\
		&
		x^5-x^4-x^3-3x^2-3x-2
		&
		\phantom{\xi}> 2.30939
	\\[1pt]
		1,1,3,4
		&
		x^3-2x^2+x-4
		&
		\phantom{\xi}> 2.31459
	\\[1pt]\hline
	\end{array}
\]
\end{footnotesize}
\caption{Short legal sequences leading to growth rates of at least $\xi$.}
\label{table-short-bad}
\end{table}

Table~\ref{table-short-bad} lists several $\preceq$-minimal legal sequences of sum indecomposable permutations that lead to growth rates equal to or greater than $\xi$. Therefore these sequences---and all sequences that dominate them---do not need to be considered in the categorization of growth rates below $\xi$. This table leaves open the sequence $1,1,2,6$ (and two sequences which dominate it, $1,1,2,6,1$ and $1,1,2,7$) but it is easy to see that these sequences cannot be realized: in order to have only $2$ sum indecomposable permutations of length $3$, a class must avoid $231$, $312$, or $321$, but each of the classes $\Av(231)$, $\Av(312)$, and $\Av(321)$ contains only $5$ sum indecomposable permutations of length $4$.

Table~\ref{table-long-bad} lists several more families of legal sequences that lead to growth rates equal to or greater than $\xi$. It differs from Table~\ref{table-short-bad} in that the sequences it presents can be arbitrarily long. Moreover, except for the first two sequences in this table, the growth rates of the corresponding classes (if the sequences can be realized, a question we do not need to consider) converge from above to $\xi$. To demonstrate this we need a basic analytic fact.

\begin{table}
\begin{footnotesize}
\[
	\begin{array}{ll}
	\hline
	\mbox{sequence}&\mbox{growth rate is the greatest real root of}
	\\\hline
	\\[-8pt]
		1, 1, 2, 3, 4^\infty
		&
		x^5-2x^4-x^2-x-1
	\\[1pt]
		1,1,2,3,4^i,5,3,3,2,1
		&
		x^5-2x^4-x^2-x-1
%		(x^{i+5}+1)\left(x^5-2x^4-x^2-x-1\right)
	\\[1pt]
		1,1,2,3,4^i,5,3,3,3
		&
		x^{i+4}\left(x^5-2x^4-x^2-x-1\right)-x^4+2x^3+3
	\\[1pt]
		1,1,2,3,4^i,5,4,1,1,1,1,1,1
		&
		x^{i+8}\left(x^5-2x^4-x^2-x-1\right)-x^8+x^7+3x^6+1
	\\[1pt]
		1,1,2,3,4^i,5,4,2
		&
		x^{i+3}\left(x^5-2x^4-x^2-x-1\right)-x^3+x^2+2x+2
	\\[1pt]
		1,1,2,3,4^i,5,5
		&
		x^{i+2}\left(x^5-2x^4-x^2-x-1\right)-x^2+5
	\\[1pt]
		1,1,2,3,4^i,6,2,1,1
		&
		x^{i+4}\left(x^5-2x^4-x^2-x-1\right)-2x^4+4x^3+x^2+1
	\\[1pt]
		1,1,2,3,4^i,6,2,2
		&
		x^{i+3}\left(x^5-2x^4-x^2-x-1\right)-2x^3+4x^2+2
	\\[1pt]
		1,1,2,3,4^i,6,3
		&
		x^{i+2}\left(x^5-2x^4-x^2-x-1\right)-2x^2+3x+3
	\\[1pt]
		1,1,2,3,4^i,7,1
		&
		x^{i+2}\left(x^5-2x^4-x^2-x-1\right)-3x^2+6x+1
	\\[1pt]
		1,1,2,3,4^i,8
		&
		x^{i+1}\left(x^5-2x^4-x^2-x-1\right)-4x+8
	\\[1pt]\hline
	\end{array}
\]
\end{footnotesize}
\caption{Long legal sequences leading to growth rates of at least $\xi$.}
\label{table-long-bad}
\end{table}

Consider for example the final sequence in this chart, $1,1,2,3,4^i,8$. If this sequence could be realized, the generating function for the sum closed class that did realize it would be
\[
	\frac{1}{1-\left(x+x^2+2x^3+3x^4+4x^5\frac{1-x^i}{1-x}+8x^{i+5}\right)}
	=
	\frac{1-x}{1-2x-x^3-x^4-x^5-4x^{i+5}+8x^{i+6}}.
\]
Proceeding as usual, the growth rate of the corresponding permutation class would be the greatest real root of
\[
	x^{i+1}\left(x^5-2x^4-x^2-x-1\right)-4x+8.
\]
In fact, all of the polynomials in Table~\ref{table-long-bad} except the first two (which both have $\xi$ as their greatest real root) are of the form $x^if(x)+g(x)$ where the greatest real root of $f(x)$ is $\xi$ and $g(x)$ is a polynomial with negative leading term. In this case we see that $g(x)$ is negative for all $x>2$ and thus the growth rates corresponding to sequences in this family converge to $\xi$ from above by the following.

\begin{proposition}
\label{prop-stupid-analysis}
Suppose that $f(x)$ and $g(x)$ are polynomials, that the largest real root of $f$ is $r>1$, and that $f'(x)$ is positive for $x>r$. Define $h_i(x)=x^if(x)+g(x)$. If there exists a $\delta>0$ such that $g$ is negative on the interval $[r,r+\delta)$ then the largest roots of the functions $h_i$ converge from above to $r$.
\end{proposition}
\begin{proof}
Because $h_i(r) < 0$ and $h_i(x) \to \infty$ as $x \to \infty$, we may conclude by the Intermediate Value Theorem that the largest real root of $h_i(x)$ is strictly greater than $r$ for all $i$.
	
To prove convergence, let $\epsilon > 0$ be given, and assume without loss of generality that $\epsilon < \delta$. We seek an integer $n_0$ such that for all $n \geq n_0$, the largest real root of $h_n(x)$ lies in the interval $(r, r+\epsilon)$. As $h_i(r) < 0$ for all $i$, it suffices to find $n_0$ such that $h_n(r + \epsilon) > 0$ for all $n \geq n_0$, or equivalently,
\[
	(r+\epsilon)^nf(r+\epsilon) + g(r+\epsilon) > 0.
\]
Moreover, to ensure that the root found in the interval $(r, r+\epsilon)$ is the largest real root, we also choose $n_0$ large enough so that $h_n'(x) > 0$ for $x>r$ for all $n \geq n_0$, or equivalently, 
\[
	x^{n-1}(xf'(x) + nf(x)) + g(x) > 0.
\]
As $f(x) > 0$ and $f'(x) > 0$ for $x>r$, such an $n_0$ can be chosen to satisfy both demands.
\end{proof}

\section{Narrowing the Search II --- The Insertion Encoding}
\label{sec-narrowing-2}

There are several families of legal sequences that would lead to growth rates less than $\xi$ yet are not realizable. We eliminate these sequences here by way of a computer search.\footnote{The \texttt{Python} code to perform the computer search can be found at the first author's website at \href{http://jaypantone.com/publications/010-xi-grs/}{http://jaypantone.com/publications/010-xi-grs/}.} The results of this section are collected below.

\begin{proposition}
\label{prop-no-1123-5}
No realizable sequence beginning with $1,1,2,3$ may contain an entry greater than $5$ before it contains an entry equal to $5$.
\end{proposition}

\begin{proposition}
\label{prop-no-112344-5-even}
No realizable sequence beginning with $1,1,2,3,4,4$ may have an even-indexed entry equal to $5$.
\end{proposition}

\begin{proposition}
\label{prop-no-112344-5-11}
No realizable sequence beginning with $1,1,2,3,4,4$ and containing a $5$ may end with $1,1$.
\end{proposition}

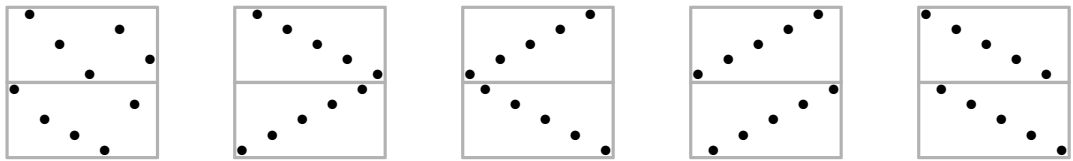
\begin{figure}
\begin{center}
	\begin{tikzpicture}[scale=0.2]
		\plotperm{5, 10, 3, 8, 2, 6, 1, 9, 4, 7};
		\plotpermbox{1}{1}{10}{5};
		\plotpermbox{1}{6}{10}{10};
	\end{tikzpicture}
\quad\quad
	\begin{tikzpicture}[scale=0.2]
		\plotperm{1, 10, 2, 9, 3, 8, 4, 7, 5, 6};
		\plotpermbox{1}{1}{10}{5};
		\plotpermbox{1}{6}{10}{10};
	\end{tikzpicture}
\quad\quad
	\begin{tikzpicture}[scale=0.2]
		\plotperm{6, 5, 7, 4, 8, 3, 9, 2, 10, 1};
		\plotpermbox{1}{1}{10}{5};
		\plotpermbox{1}{6}{10}{10};
	\end{tikzpicture}
\quad\quad
	\begin{tikzpicture}[scale=0.2]
		\plotperm{6, 1, 7, 2, 8, 3, 9, 4, 10, 5};
		\plotpermbox{1}{1}{10}{5};
		\plotpermbox{1}{6}{10}{10};
	\end{tikzpicture}
\quad\quad
	\begin{tikzpicture}[scale=0.2]
		\plotperm{10, 5, 9, 4, 8, 3, 7, 2, 6, 1};
		\plotpermbox{1}{1}{10}{5};
		\plotpermbox{1}{6}{10}{10};
	\end{tikzpicture}
\end{center}
\caption{From left to right, an arbitrary vertical alternation, two vertical wedge alternations, and two vertical parallel alternations.}
\label{fig-vert-alternation}
\end{figure}

We establish Propositions~\ref{prop-no-1123-5}--\ref{prop-no-112344-5-11} using an implementation of the \emph{insertion encoding}, which is a length-preserving bijection between permutation classes and formal languages introduced by Albert, Linton, and Ru\v{s}kuc~\cite{albert:the-insertion-e:}. Importantly, there is a characterization of the classes whose insertion encodings are regular languages. To state this characterization we need a definition: a \emph{vertical alternation} is a permutation in which every entry of odd index lies above every entry of even index, or the complement of such a permutation. For example, the permutations shown in Figure~\ref{fig-vert-alternation} are all vertical alternations. These are precisely the obstructions to having a regular insertion encoding:

\begin{theorem}[Albert, Linton, and Ru\v{s}kuc~\cite{albert:the-insertion-e:}]
\label{thm-insertion-encoding}
The insertion encoding of a permutation class $\C$ forms a regular language if and only if $\C$ does not contain arbitrarily long vertical alternations.
\end{theorem}

It follows readily from the Erd\H{o}s--Szekeres Theorem that a permutation class contains arbitrarily long vertical alternations if and only if it contains arbitrarily long vertical alternations of one of the four types shown on the right of Figure~\ref{fig-vert-alternation}. We call these vertical \emph{wedge} alternations and vertical \emph{parallel} alternations.

In order to prove Propositions~\ref{prop-no-1123-5}--\ref{prop-no-112344-5-11}, we need to consider classes whose sequences of sum indecomposable permutations begin with $1,1,2,3$. It is easy to see that every permutation class that contains arbitrarily long vertical wedge or parallel alternation contains at least $4$ sum indecomposable permutations of length $4$, and thus the classes we are interested in have regular insertion encodings.

Moreover, an algorithm is described in Vatter~\cite{vatter:finding-regular:}---using the Myhill--Nerode Theorem~\cite{myhill:finite-automata:,nerode:linear-automato:}---which can compute the generating functions of classes with regular insertion encodings and the generating functions enumerating sum indecomposable permutations in such classes. An implementation of this algorithm is available in a Python package developed by Homberger and Pantone~\cite{PermPy1.0}. We use this algorithm to inspect all potential counterexamples to Propositions~\ref{prop-no-1123-5}--\ref{prop-no-112344-5-11} (and eventually conclude that there are none).

We begin by examining the classes to which Proposition~\ref{prop-no-1123-5} applies. In order for the sequence of sum indecomposable permutations to begin with $1, 1, 2$, the class must avoid one of the $3$ sum indecomposable permutations of length $3$. In each of the three resulting classes---$\Av(231)$, $\Av(312)$, and $\Av(321)$---there are $5$ sum indecomposable permutations of length $4$. In order for the sequence to begin with $1,1,2,3$ we must choose $2$ of these $5$ to add to the basis. 

At this point we have $30$ permutation classes whose sequences of sum indecomposable permutations begin with $1,1,2,3$. For each such class $\C=\Av(B)$ we create a list of the sum indecomposable permutations in the class up to length $10$. If these initial terms do not contain an entry greater than $5$ then we use the insertion encoding to verify that in fact the sequence enumerating sum indecomposable members of the class \emph{never} has an entry greater than $5$. If the initial terms do contain an entry of value greater than $5$ following another entry of value $5$, letting $i$ denote the length corresponding to occurrence of value $5$, we repeat the test on each class whose basis consists of $B$ together with one of the sum indecomposable permutations of length $i$ in $\C$. Obviously, if we encounter a sequence that contains an entry of value greater than $5$ that does not follow a $5$, we have found a counterexample.

For example, one of the $30$ initial classes is $\C = \Av(231, 4312, 4321)$, and the enumeration of the sum indecomposable members of this class is given by the Fibonacci numbers $1, 1, 2, 3, 5, 8, 13, 21, \dots$. The $5$ sum indecomposable permutations of length $5$ are $51234$, $51243$, $51324$, $52134$, and $52143$. To attempt to find a counterexample to Proposition~\ref{prop-no-1123-5} we add each of these to the basis one at a time, bringing the number of sum indecomposable permutations of length $5$ down to $4$, leading us to $5$ new classes to check.

Continuing this example, when we add $51234$ to the basis, the result is a class with a finite sequence of sum indecomposable permutations, $1,1,2,3,4,3,1$. Adding either $51243$, $51324$, or $52134$ to the basis instead yields a class with sum indecomposable sequence $1, 1, 2, 3, 4^i$. More interestingly, the class $\Av(231,4312,4321,52143)$ has sum indecomposable sequence $1, 1, 2, 3, 4, 5, 6, 7, \dots$, and thus its subclasses require further examination. Repeating the process, we form each class whose basis consists of $\{231, 4312, 4321, 52143\}$ together with one of the $5$ sum indecomposable permutations of length $6$. As none of these $5$ subclasses contains $5$ or more sum indecomposable permutations of the same length, none of them are counterexamples to the statement. Therefore, all subclasses of $\Av(231, 4312, 4321)$ satisfy Proposition~\ref{prop-no-1123-5}.

This process terminates after the insertion encoding has been applied to $178$ permutation classes. We then use the insertion encoding to compute the generating function for the sum indecomposable members of each of these classes. In every case the sequence is either finite or ends with an infinite repeating sequence of the same number. For example, the generating function enumerating sum indecomposable members of $\Av(231,4312,4321,52143,613245)$ is computed as
\[
	x + x^2 + 2x^3 + 3x^4 + 4\frac{x^5}{1-x}.
\]
Upon verifying that none of the $178$ generating functions obtained in this manner contains a term greater than $5$, the proof of Proposition~\ref{prop-no-1123-5} is complete.

The proofs of Propositions~\ref{prop-no-112344-5-even} and \ref{prop-no-112344-5-11} are related to each other and similar to that of Proposition~\ref{prop-no-1123-5}. We start with all classes whose bases consist of permutations of length $6$ or less and whose sequences of sum indecomposable members begin $1,1,2,3,4,4$. There are $173$ such classes, and we find via the insertion encoding that only $2$ of them contain $5$ sum indecomposable permutations of the same length, namely
\[
	\Av(321, 3412, 4123, 23451, 314625)
	\quad\text{and}\quad
	\Av(321, 2341, 3412, 51234, 251364).
\]
Note that these classes are inverses of each other. We again use the insertion encoding to compute that the sequence of sum indecomposable permutations in each of these two classes, which is
\[
	1,1,2,3,4,4,5,4,5,4,\dots.
\]
As neither class contains an even-indexed entry equal to $5$, we have proved Proposition~\ref{prop-no-112344-5-even}.

It follows by direct observation that $\bigoplus\Sub(U^o)$ lies in the first of these classes, while its inverse class lies in the second. As we observed in Section~\ref{sec-xi}, the sequence of sum indecomposable permutations in $\bigoplus\Sub(U^o)$ is also $1,1,2,3,4,4,5,4,5,4,\dots$, so we immediately obtain the following result which we use in the next section.

\begin{proposition}
\label{prop-basis-sumU}
The basis of the class $\bigoplus\Sub(U^o)$ consists of $321$, $3412$, $4123$, $23451$, and $314625$.
\end{proposition}

From inspection, we see that every sum indecomposable permutation of length at least $4$ in this class (or its inverse) has at least $2$ sum indecomposable children. Thus no subclass of $\bigoplus\Sub(U^o)$ may have its sequence of sum indecomposable members end with $1,1$, proving Proposition~\ref{prop-no-112344-5-11}.

\section{Realizing the Remaining Sequences}
\label{sec-realizing}

We now show that all remaining legal sequences can be realized, which we do with two constructions. The first of these consists of four chains of sum indecomposable permutations together with $5$ sporadic sum indecomposable permutations. Specifically, the four chains are of the following forms:
\begin{itemize}
	\item Type $1$: $12\cdots (n-1)\ominus 1$ for $n\ge 1$,
	\item Type $2$: $(21\oplus 12\cdots (n-3))\ominus 1$ for $n\ge 3$,
	\item Type $3$: $(1\oplus 21\oplus 12\cdots (n-4))\ominus 1$ for $n\ge 4$, and
	\item Type $4$: $(12\oplus 21\oplus 12\cdots (n-5))\ominus 1$ for $n\ge 5$.
\end{itemize}
The sporadic elements are $312$, $3421$, $4321$, $(21\oplus 21)\ominus 1=32541$, and $(123\oplus 21)\ominus 1=234651$.

\begin{figure}
\begin{center}
	\begin{footnotesize}
	\begin{tikzpicture}[xscale=3, yscale=0.9]
		\node at (0,8.5) {\normalsize Type 1};
		\node at (1,8.5) {\normalsize Type 2};
		\node at (2,8.5) {\normalsize Type 3};
		\node at (3,8.5) {\normalsize Type 4};
		\node (01) at (0,1) {$1$};
		\node (02) at (0,2) {$21$};
		\node (03) at (0,3) {$12\ominus 1$};
		\node (04) at (0,4) {$123\ominus 1$};
		\node (05) at (0,5) {$1234\ominus 1$};
		\node (06) at (0,6) {$12345\ominus 1$};
		\node (07) at (0,7) {$123456\ominus 1$};
		\node (08) at (0,8) {$\vdots$};
		\draw (01)--(02)--(03)--(04)--(05)--(06)--(07)--(08);
		\node (13) at (1,3) {$21\ominus 1$};
		\node (14) at (1,4) {$(21\oplus 1)\ominus 1$};
		\node (15) at (1,5) {$(21\oplus 12)\ominus 1$};
		\node (16) at (1,6) {$(21\oplus 123)\ominus 1$};
		\node (17) at (1,7) {$(21\oplus 1234)\ominus 1$};
		\node (18) at (1,8) {$\vdots$};
		\draw (02)--(13)--(14)--(15)--(16)--(17)--(18);
		\draw (03)--(14);
		\draw (04)--(15);
		\draw (05)--(16);
		\draw (06)--(17);
		\draw (07)--(18);
		\node (312) at (2,3) {$312$};
		\draw (02)--(312);
		\node (24) at (2,4) {$(1\oplus 21)\ominus 1$};
		\node (25) at (2,5) {$(1\oplus 21\oplus 1)\ominus 1$};
		\node (26) at (2,6) {$(1\oplus 21\oplus 12)\ominus 1$};
		\node (27) at (2,7) {$(1\oplus 21\oplus 123)\ominus 1$};
		\node (28) at (2,8) {$\vdots$};
		\draw (24)--(25)--(26)--(27)--(28);
		\draw (28)--(07);
		\draw (27)--(06);
		\draw (26)--(05);
		\draw (25)--(04);
		\draw (24)--(03);
		\draw (28)--(17);
		\draw (27)--(16);
		\draw (26)--(15);
		\draw (25)--(14);
		\draw (24)--(13);
		\node (3421) at (3,4) {$3421$};
		\draw (3421)--(03);
		\draw (3421)--(13);
		\node (4321) at (4,4) {$4321$};
		\draw (4321)--(13);
		\node (35) at (3,5) {$(12\oplus 21)\ominus 1$};
		\node (36) at (3,6) {$(12\oplus 21\oplus 1)\ominus 1$};
		\node (37) at (3,7) {$(12\oplus 21\oplus 12)\ominus 1$};
		\node (38) at (3,8) {$\vdots$};
		\draw (35)--(36)--(37)--(38);
		\draw (38)--(27);
		\draw (38)--(07);
		\draw (37)--(26);
		\draw (37)--(06);
		\draw (36)--(25);
		\draw (36)--(05);
		\draw (35)--(24);
		\draw (35)--(04);
		\node (45) at (4,5) {$(21\oplus 21)\ominus 1$};
		\draw (45)--(14);
		\draw (45)--(24);
		\node (46) at (4,6) {$(123\oplus 21)\ominus 1$};
		\draw (46)--(35);
	\end{tikzpicture}
	\end{footnotesize}
\end{center}
\caption{A set of sum indecomposable permutations that can be used to realize any legal sequence dominated by $1,1,3,5,5,5,4^\infty$.}
\label{fig-master-construction}
\end{figure}
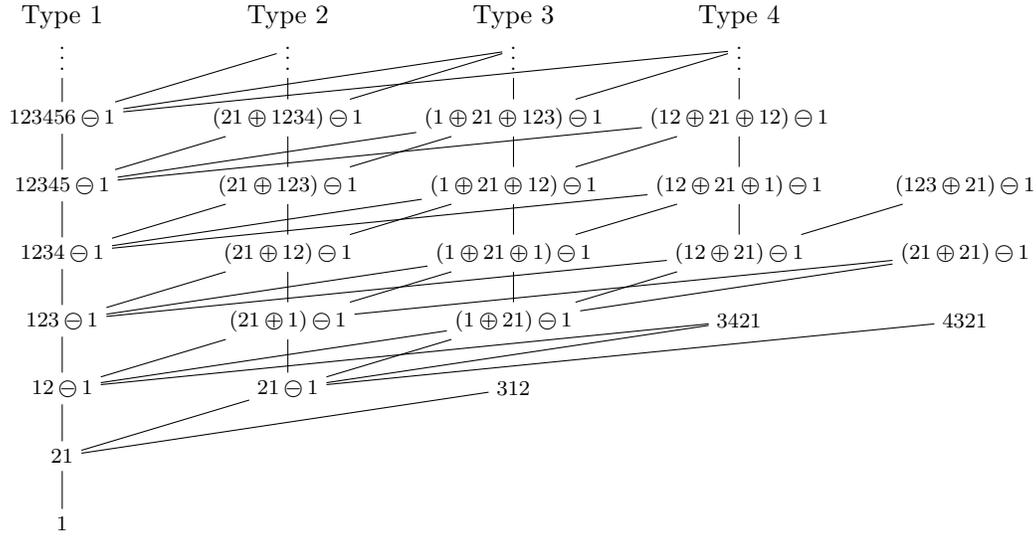

The Hasse diagram displaying the containment relations between these permutations is shown in Figure~\ref{fig-master-construction}. Importantly, all containment relations in this diagram run up and to the right. In particular, elements of type $j$ only contain elements of type $i$ for $i\le j$. Therefore any legal sequence $(s_n)$ dominated by $1,1,3,5,5,5,4^\infty$ can be realized by taking the sum indecomposable permutations of length $n$ to be the leftmost $s_n$ permutations on level $n$ of Figure~\ref{fig-master-construction}.

Moreover, we can show that these sequences are realized by finitely based classes as follows. It can be computed that the sum indecomposable permutations in this construction are all contained in the sum closed class whose finite basis consists of $2413$, $3142$, $3412$, $4123$, $4132$, $4213$, $4231$, $4312$, $24531$, $25431$, $34251$, $34521$, $35421$, $43251$, $43521$, $45321$, $54321$, $243651$, $324651$, $325461$, $2345761$, and $2346571$. Using the insertion encoding, it can then be verified that the sum indecomposable permutations in this class have the same enumeration as those of our construction, and thus this is the basis for the sum closure of all permutations used in the construction. The four types of sum indecomposable permutations in this class form chains, and thus we need only add to this basis the shortest excluded permutation of each type. This yields the following result.

\begin{proposition}
\label{prop-realize-legal}
All legal sequences of sum indecomposable permutations that are dominated by $1,1,3,5,5,5,4^\infty$ are realizable. Moreover, all of these sequences are realized by finitely based classes.
\end{proposition}

While the construction used to prove Proposition~\ref{prop-realize-legal} might seem arbitrary, it is far from it. Up to symmetry there is only one class whose sequence of sum indecomposable permutations begins $1,1,2,3,5$. That class is $\Av(312,4321,3412)$, and its sum indecomposable permutations are all of the form $(\bigoplus\{1,21\})\ominus 1$. Because we need to realize sequences that begin with $1,1,2,3,5$, our choice of the four chains was essentially forced.

Our second construction uses subclasses of $\bigoplus\Sub(U^o)$.

\begin{proposition}
\label{prop-realize-long-legal}
All legal sequences of sum indecomposable permutations that are dominated by $1,1,2,3,4^{2i},5,4^\infty$ and do not end in $1, 1$ are realizable. Moreover, all of these sequences that lead to a growth rate of less than $\xi$ are realized by finitely based classes.
\end{proposition}
\begin{proof}
Let $(s_n)$ be a sequence of this form. If $(s_n)$ does not contain a $5$ then it is realizable by Proposition~\ref{prop-realize-legal}. Thus (by the results of Section~\ref{sec-taper}) we may assume that $(s_n)$ begins $1,1,2,3,4^{2i},5$, and ends with zero or more $4$s, followed by zero or more $3$s, followed by zero or more $2$s, and then possibly a single $1$ (though any one of these portions of $(s_n)$ other than the $1$s might be infinite). We realize such a sequence by taking the sum indecomposable members $\PropSub(U^o)$ of lengths up to $2i+5$ (giving a sequence beginning with $1,1,2,3,4^{2i+1}$) together with $\mu_{2i+5}$ itself, which gives us $5$ sum indecomposable permutations of length $2i+5$.

Later $4$s in $(s_n)$ can be realized with members of $\PropSub(U^o)$, later $3$s with the two increasing oscillations of each length together with heads of members of $U^o$ (increasing oscillations of primary type with their first entries inflated), and later $2$s simply with the two increasing oscillations of each length. As we have assumed that $(s_n)$ does not end in $1,1$, it can end in at most one $1$, and for this sum indecomposable permutation we can take either increasing oscillation.

It remains to show that the sequences of this form that lead to growth rates under $\xi$ are realized by finitely based classes. First recall that $\xi$ is the growth rate of the sum closed class whose sequence of sum indecomposable permutations is $1,1,2,3,4^\infty$. Therefore every sequence of the form $1,1, 2,3,4^{2i},5,4^\infty$ leads to a growth rate greater than $\xi$. The sequences leading to growth rates less than $\xi$ therefore either do not contain a $5$ (and thus are included in Proposition~\ref{prop-realize-legal}) or do not end in $4^\infty$.

By Proposition~\ref{prop-basis-sumU} we have
\[
	\bigoplus\Sub(U^o)
	=
	\Av(321, 3412, 4123, 23451, 314625).
\]
Therefore the basis of $\bigoplus\PropSub(U^o)$ consists of the five basis elements of $\bigoplus\Sub(U^o)$ together with $U^o$ itself. We are only interested in sequences that include a term less than $4$ after their $5$, so our construction above instructs us to omit all sufficiently long tails (increasing oscillations with their greatest entries inflated). This can be imposed by a single additional restriction. Moreover, once this tail is omitted, we do not need the members of $U^o$ that are longer than this tail in our basis. From that point, depending on how the sequence tapers, we need add at most three additional basis elements---a head, and one or two of the increasing oscillations---establishing that these classes all have finite bases.
\end{proof}

Again, the construction used to prove Proposition~\ref{prop-realize-long-legal} is far from arbitrary. The proofs of Propositions~\ref{prop-no-112344-5-even} and \ref{prop-no-112344-5-11} show that our only choice is whether to use the antichain $U^o$ or its inverse, which permutation to use in addition to increasing oscillations to realize any $3$s in the sequence (i.e., to choose between heads and tails), and then which of the two increasing oscillations to use to realize the final $1$ of a sequence, if it ends with a $1$.

\begin{table}
\begin{footnotesize}
\[
	\begin{array}{lllr}
	\hline
	\mbox{sequence}&\mbox{restriction}&\mbox{growth rate is the greatest real root of}&\mbox{growth rate}
	\\\hline
	\\[-8pt]
		1, 1, 3, 3, 1^i
		&
		i\le 5
		&
		x^i (x^5 - 2x^4 - 2x^2 + 2) + 1
%		x^i(x^{4} - x^{3} - x^{2} - 3 x - 3)-x^{i-1}-\cdots-x-1
		&
		\lesssim 2.30503
	\\
		1, 1, 3, 2^\infty
		&
		&
		x^{4} - 2 x^{3} - 2 x + 1
		&
		\approx 2.29663
	\\
		1, 1, 3, 2^i, 1^\infty
		&
		&
		x^{i} \left(x^{4} - 2 x^{3} - 2 x + 1\right) + 1
		&
		\to 2.29663
	\\
		1, 1, 3, 2^i, 1^j
		&
		&
		x^{i + j} \left(x^{4} - 2 x^{3} - 2 x + 1\right) + x^j + 1
		&
		\to 2.29663
	\\
		1, 1, 2, 5, 2, 1
		&
		&
		x^{6} - x^{5} - x^{4} - 2 x^{3} - 5 x^{2} - 2 x - 1
		&
		\approx 2.30490
	\\
		1, 1, 2, 5, 2
		&
		&
		x^{5} - x^{4} - x^{3} - 2 x^{2} - 5 x - 2
		&
		\approx 2.29783
	\\
		1, 1, 2, 5, 1^\infty
		&
		&
		x^{5} - 2 x^{4} - x^{2} - 3 x + 4
		&
		\approx 2.29408
	\\
		1, 1, 2, 5, 1^i
		&
		&
		x^{i} \left(x^{5} - 2 x^{4} - x^{2} - 3 x + 4\right)+1
		&
		\to 2.29408
	\\
		1,1,2,4,4,1^i
		&
		i\le 5
		&
		x^i(x^6 - 2x^5 - x^3 - 2x^2 + 3) + 1
%		x^i(x^{5} - x^{4} - x^{3} - 2 x^{2} - 4 x - 4)-x^{i-1}-\cdots-x-1
		&
		\lesssim 2.30515
	\\
		1,1,2,4,3,3,2
		&
		&
		x^{7} - x^{6} - x^{5} - 2 x^{4} - 4 x^{3} - 3 x^{2} - 3 x - 2
		&
		\approx 2.30394
	\\
		1,1,2,4,3,3,1^\infty
		&
		&
		x^{7} - 2 x^{6} - x^{4} - 2 x^{3} + x^{2} + 2
		&
		\approx 2.30326
	\\
		1,1,2,4,3, 3, 1^i
		&
		&
		x^{i} \left(x^{7} - 2 x^{6} - x^{4} - 2 x^{3} + x^{2} + 2\right)+1
		&
		\to 2.30326
	\\
		1,1,2,4,3,2^\infty
		&
		&
		x^{6} - 2 x^{5} - x^{3} - 2 x^{2} + x + 1
		&
		\approx 2.30167
	\\
		1,1,2,4,3,2^i,1^\infty
		&
		&
		x^{i} \left(x^{6} - 2 x^{5} - x^{3} - 2 x^{2} + x + 1\right)+1
		&
		\to 2.30167
	\\
		1,1,2,4,3,2^i,1^j
		&
		&
		x^{i + j} \left(x^{6} - 2 x^{5} - x^{3} - 2 x^{2} + x + 1\right)+x^{j}+1
		&
		\to 2.30167
	\\
		1,1,2,4,2^\infty
		&
		&
		x^{5} - 2 x^{4} - x^{2} - 2 x + 2
		&
		\approx 2.28563
	\\
		1,1,2,4,2^i,1^\infty
		&
		&
		x^{i} \left(x^{5} - 2 x^{4} - x^{2} - 2 x + 2\right)+1
		&
		\to 2.28563
	\\
		1,1,2,4,2^i,1^j
		&
		&
		x^{i + j} \left(x^{5} - 2 x^{4} - x^{2} - 2 x + 2\right)+x^{j}+1
		&
		\to 2.28563
	\\
		1,1,2^\infty
		&
		&
		x^{3} - 2 x^{2} - 1
		&
		\approx 2.20557
	\\
		1,1,2^i,1^\infty
		&
		&
		x^{i} \left(x^{3} - 2 x^{2} - 1\right)+1
		&
		\to 2.20557
	\\
		1,1,2^i,1^j
		&
		&
		x^{i + j} \left(x^{3} - 2 x^{2} - 1\right)+x^{j}+1
		&
		\to 2.20557
	\\
		1^\infty
		&
		&
		x-2
		&
		= 2\phantom{.00000}
	\\
		1^i
		&
		&
		x^i(x-2) + 1
		&
		\to 2\phantom{.00000}
	\\[1pt]\hline
	\end{array}
\]
\end{footnotesize}
\caption{Legal realizable sequences dominated by a sequence in Table~\ref{table-short-bad} leading to growth rates under $\xi$.}
\label{table-short-good-realizable}
\end{table}

We are now ready to prove our main result.

\begin{theorem}
\label{thm-xi-grs-main}
The set of growth rates of permutation classes below $\xi$ can be characterized by a finite set of infinite families of algebraic numbers, collected in Tables~\ref{table-short-good-realizable} and \ref{table-long-good-realizable}.
\end{theorem}
\begin{proof}
Taking into account the $\preceq$-minimal sequences that lead to growth rates equal to or greater than $\xi$ from Tables~\ref{table-short-bad} and \ref{table-long-bad}, we can describe the complete list of growth rates of permutation classes smaller than $\xi$. First, Table~\ref{table-short-good-realizable} lists all realizable sequences that are dominated by one of the short sequences from Table~\ref{table-short-bad}; by Proposition~\ref{prop-realize-legal} all of the candidate sequences are realizable except $1,1,2,6$ (as established in Section~\ref{sec-narrowing-1}).

Not all sequences dominated by one of the long sequences in Table~\ref{table-long-bad} are realizable. All that begin with $1,1,2,3,4^i,5$ and end with $1,1$ (or $1^j$ or $1^\infty$) can be realized for $i=1$ by Proposition~\ref{prop-realize-legal}, but cannot be realized for $i\ge 2$ by Proposition~\ref{prop-no-112344-5-11}. In addition, Proposition~\ref{prop-no-112344-5-even} shows that sequences beginning $1,1,2,3,4^i,5$ for $i\ge 2$ are realizable only if $i$ is even.

All other sequences that are dominated by one of the long sequences in Table~\ref{table-long-bad} can be realized by Proposition~\ref{prop-realize-long-legal}. This produces the list of realizable legal sequences in Table~\ref{table-long-good-realizable}, and completes the list of all growth rates of permutation classes at most $\xi$. 
\end{proof}

\begin{table}
\begin{footnotesize}
\[
	\begin{array}{lll}
	\hline
	\mbox{sequence}&\mbox{restriction}&\mbox{growth rate is the greatest real root of}
	\\\hline\\[-8pt]
		1,1,2,3,4^i,5,4,1^j
		&
		i\le 1, \, j\le 5
		&
		x^{i + j + 2} \left(x^{5} - 2 x^{4} - x^{2} - x - 1\right) - x^{j}\left(x^{2} - x - 3\right)+1
	\\
		1,1,2,3,4^i,5,4,1^j
		&
		i\text{ even}, \, j\le 1
		&
		x^{i + j + 2} \left(x^{5} - 2 x^{4} - x^{2} - x - 1\right) - x^{j} \left(x^{2} - x - 3\right)+1
	\\
		1,1,2,3,4^i,5,3,3,2
		&
		i = 1 \text{ or }i\text{ even}
		&
		x^{i+4} \left(x^{5} - 2 x^{4} - x^{2} - x - 1\right) - x^{4} + 2 x^{3} + x + 2
	\\
		1,1,2,3,4^i,5,3,3,1^\infty
		&
		i\le 1
		&
		x^{i+3} \left(x^{5} - 2 x^{4} - x^{2} - x - 1\right) - x^{3} + 2 x^{2} + 2
	\\
		1,1,2,3,4^i,5,3,3,1^j
		&
		i\le 1
		&
		x^{i + j+ 3} \left(x^{5} - 2 x^{4} - x^{2} - x - 1\right) - x^{j} \left(x^{3} - 2 x^{2} - 2\right)+1
	\\
		1,1,2,3,4^i,5,3,3,1^j
		&
		i\text{ even}, \, j\le 1
		&
		x^{i + j+ 3} \left(x^{5} - 2 x^{4} - x^{2} - x - 1\right) - x^{j} \left(x^{3} - 2 x^{2} - 2\right)+1

	\\
		1,1,2,3,4^i,5,3^j,2^\infty
		&
		i = 1 \text{ or }i\text{ even}, \, j \leq 1
		&
		x^{i + j + 1} \left(x^{5} - 2 x^{4} - x^{2} - x - 1\right) - x^{j} \left(x - 2\right)+1
	\\
		1,1,2,3,4^i,5,3^j,2^k,1^\infty
		&
		i\le 1, \, j \leq 1
		&
		x^{i + j + k + 1} \left(x^{5} - 2 x^{4} - x^{2} - x - 1\right) - x^{j + k} \left(x - 2\right)+x^{k}+1
	\\
		1,1,2,3,4^i,5,3^j,2^k,1^\ell
		&
		i\le 1, \, j \leq 1
		&
		x^{i + j + k + \ell + 1} \left(x^{5} - 2 x^{4} - x^{2} - x - 1\right)
	\\
		&&\quad
			\phantom{.}- x^{j + k + \ell} \left(x - 2\right)+x^{k + \ell}+x^{\ell}+1
	\\
		1,1,2,3,4^i,5,3^j,2^k,1^\ell
		&
		i\text{ even}, \, j \leq 1, \, \ell \leq 1
		&
		x^{i + j + k + \ell + 1} \left(x^{5} - 2 x^{4} - x^{2} - x - 1\right)
	\\
		&&\quad
			\phantom{.}- x^{j + k + \ell} \left(x - 2\right)+x^{k + \ell}+x^{\ell}+1
	\\
		1,1,2,3,4^i,3^\infty
		&
		&
		x^{i} \left(x^{5} - 2 x^{4} - x^{2} - x - 1\right)+1
	\\
		1,1,2,3,4^i,3^j,2^\infty
		&
		&
		x^{i + j} \left(x^{5} - 2 x^{4} - x^{2} - x - 1\right)+x^{j}+1
	\\
		1,1,2,3,4^i,3^j,2^k,1^\infty
		&
		&
		x^{i + j + k} \left(x^{5} - 2 x^{4} - x^{2} - x - 1\right)+x^{j + k}+x^{k}+1
	\\
		1,1,2,3,4^i,3^j,2^k,1^\ell
		&
		&
		x^{i + j + k + \ell} \left(x^{5} - 2 x^{4} - x^{2} - x - 1\right)+x^{j + k + \ell}+x^{k + \ell}+x^{\ell}+1
	\\[1pt]\hline
	\end{array}
\]
\end{footnotesize}
\caption{More legal sequences. Variables that are not specified are allowed to be arbitrary nonnegative integers. The sequence of growth rates in each row of this table for which $i$ is not bounded converge to $\xi$ as $i\rightarrow\infty$.}
\label{table-long-good-realizable}
\end{table}

\section{Finite Bases}
\label{sec-fb}

The only major result we have left to establish is the following.

\begin{theorem}
\label{thm-xi-fb}
Every growth rate of a permutation class less than or equal to $\xi$ is achieved by a finitely based class.	
\end{theorem}

As our previous work (in particular Propositions~\ref{prop-realize-legal} and \ref{prop-realize-long-legal}) shows that all upper growth rates of permutation classes that are \emph{less than} $\xi$ can be achieved by finitely based classes, it remains only to show that $\xi$ itself can be realized by a finitely based class. In fact, the quintessential class of growth rate $\xi$, $\bigoplus\PropSub(U^o)$, is infinitely based. However, there is also a finitely based class with this growth rate. Indeed, we have already seen its sequence of sum indecomposable members, $1,1,2,4,3,3,2,1$, in the top row of Table~\ref{table-short-bad}. The class itself is
\[
	\Av(231, 4132, 4213, 54312, 7612345, 81234567, 987654321).
\]
This class is sum closed because all of its basis elements are sum indecomposable. To verify that it has the claimed sequence of sum indecomposable permutations it suffices to list all sum indecomposable members of this class up to length $9$. Doing so verifies that the sequence begins $1,1,2,4,3,3,2,1,0$, and because the class contains no sum indecomposable permutations of length $9$, it cannot contain any of a longer length (one could also use the insertion encoding to verify this). A routine computation shows that this class has growth rate $\xi$, finishing the proof of Theorem~\ref{thm-xi-fb}.

As we showed in Section~\ref{sec-xi}, there are uncountably many growth rates in every neighborhood of $\xi$, so almost all of these are not achievable by finitely based classes. Therefore $\xi$ represents the phase transition where the sets of growth rates of all permutation classes and that of finitely based classes first differ.

Our results also imply that $\xi$ is the least accumulation point of growth rates from above. The fact that the set of growth rates contains accumulation points from above---first established by Albert and Linton~\cite{albert:growing-at-a-pe:}---disproved a conjecture of Balough, Bollob\'as, and Morris~\cite{balogh:hereditary-prop:ordgraphs} which stated (in the more general context of hereditary properties of ordered graphs) that the set of growth rates could not have such accumulation points. Klazar~\cite{klazar:overview-of-som:}, who denotes the set of growth rates of permutation classes by $E$, suggested a possible way to revive their conjecture, writing that ``it seems that the refuted conjectures should have been phrased for finitely based downsets.'' In particular, his Problem 2.7 reads
\begin{quote}
	Let $E^\ast$ be the countable subset of $E$ consisting of the growth [rates] of finitely based [permutation classes]. Show that every $\alpha$ in $E^\ast$ is an algebraic number and that for every $\alpha$ in $E^\ast$ there is a $\delta>0$ such that $(\alpha,\alpha+\delta)\cap E^\ast=\emptyset$.	
\end{quote}
While the first part of his problem regarding the algebraicity of growth rates of finitely based permutation classes remains open, we conclude by providing a counterexample to the second part.

Recall from Proposition~\ref{prop-basis-sumU} that
\[
	\bigoplus\Sub(U^o)
	=
	\Av(321, 3412, 4123, 23451, 314625).
\]
Therefore each of the classes $\bigoplus\left( \Sub(U^o) \setminus \{\mu_7, \dots, \mu_{2i+7}\} \right)$ has a finite basis,
\[
	\{321, 3412, 4123, 23451, 314625, \mu_7, \dots, \mu_{2i+7}\}.
\]
Moreover, these classes have generating functions of the form
\[
	\frac{1}{1-\left(x+x^2+2x^3+3x^4+4\frac{x^5}{1-x}+\frac{x^{2i+7}}{1-x^2}\right)}
	=
	\frac{1-x^2}{1-x-2x^2-x^3-2x^4-2x^5-x^6-x^{2i+7}},
\]
and consequently their growth rates are the greatest real roots of the polynomials
\begin{multline*}
	% Maple: x^(2*i+5)-x^(2*i+4)-2*x^(2*i+3)-x^(2*i+2)-2*x^(2*i+1)-2*x^(2*i)-x^(2*i-1)-1
	x^{2i+7}-x^{2i+6}-2x^{2i+5}-x^{2i+4}-2x^{2i+3}-2x^{2i+2}-x^{2i+1}-1\\
	=
	(x^5-2x^4-x^2-x-1)(x+1)x^{2i+1}-1.
\end{multline*}
Thus these growth rates are solutions to
\[
	x^5-2x^4-x^2-x-1
	=
	\frac{1}{(x+1)x^{2i+1}},
\]
and it follows that these growth rates accumulate at $\xi$ from above. Combining this with our observation that there is a finitely based class with growth rate $\xi$, we obtain the following.

\begin{proposition}
\label{prop-finitely-based-xi}
There is a sequence of finitely based permutation classes whose growth rates accumulate at $\xi$ from above. Moreover, $\xi$ is itself the growth rate of a finitely based permutation class.
\end{proposition}

\section{Concluding Remarks}
\label{sec-conclusion}

In \cite{vatter:growth-rates-of:} it was proved that $\xi\approx 2.30522$ represents the phase transition from countably to uncountably many growth rates of permutation classes and that every growth rate below $\xi$ is achieved by a sum closed class. Here we have used this result to determine the complete list of growth rates below $\xi$, establishing that $\xi$ is the least accumulation point of growth rates from above, and showing that each of these growth rates is achieved by a finitely based class, while there are growth rates arbitrarily close to $\xi$ which cannot be achieved by finitely based classes.

Given that Bevan~\cite{bevan:intervals-of-pe:} has shown that every real number at least $\lambda_B\approx 2.35698$ is the growth rate of a permutation class, it is natural to try to close this gap of approximately $0.05176$. The first challenge would be to extend the results of \cite{vatter:growth-rates-of:} to this range, i.e., to show that every growth rate between $\xi$ and $\lambda_B$ is achieved by a sum closed class. More generally, \cite{vatter:growth-rates-of:} presents a conjecture that every growth rate of a permutation class is achieved by a sum closed class; this is known to be true for growth rates up to $\xi$ by \cite{vatter:growth-rates-of:} and for growth rates between $\lambda_B$ and approximately $3.79$ by \cite{bevan:intervals-of-pe:,vatter:permutation-cla:lambda:}.

Supposing that this conjecture were established, extending the results of this paper up to $\lambda_B$ would still appear to be a difficult task. For a start, determining the realizable sequences of sum indecomposable permutations would require much more nuance than was required in Sections~\ref{sec-taper}--\ref{sec-realizing}. While we only had to use one infinite antichain ($U^o$, in Proposition~\ref{prop-realize-long-legal}) to realize the sequences below $\xi$, one would need to consider several infinite antichains (some of which are presented in the conclusion of \cite{vatter:growth-rates-of:}) to extend this work to $\lambda_B$.

Indeed, even expressing the realizable sequences in this range is more challenging. Recall that when we considered sum closed classes lying between $\bigoplus\PropSub(U^o)$ and $\bigoplus\Sub(U^o)$ in Section~\ref{sec-xi}, we were able to say that these classes had sequences $(s_n)$ of sum indecomposable members for all sequences $(s_n)$ satisfying
\[
	(1,1,2,3,4^\infty)\preceq (s_n)\preceq (1,1,2,3,4,4,5,4,5,4,\dots).
\]

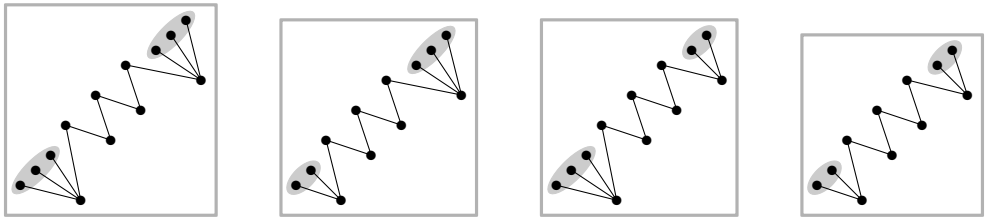
\begin{figure}
\begin{center}
	\begin{tikzpicture}[scale=0.2]
		\draw[lightgray, fill, rotate around={-45:(2,3)}] (2,3) ellipse (20pt and 60pt);
		\draw[lightgray, fill, rotate around={-45:(11,12)}] (11,12) ellipse (20pt and 60pt);
		\plotpermgraph{2,3,4,6,1,8,5,10,7,11,12,13,9};
		\plotpermbox{0.5}{0.5}{13.5}{13.5};
	\end{tikzpicture}
\quad\quad
	\begin{tikzpicture}[scale=0.2]
		% Intervals:
		\draw[lightgray, fill, rotate around={-45:(1.5,2.5)}] (1.5,2.5) ellipse (20pt and 40pt);
		\draw[lightgray, fill, rotate around={-45:(10,11)}] (10,11) ellipse (20pt and 60pt);
		% The permutation:
		\plotpermbox{0.5}{0.5}{12.5}{12.5};
		\plotpermgraph{2,3,5,1,7,4,9,6,10,11,12,8};
	\end{tikzpicture}
\quad\quad
	\begin{tikzpicture}[scale=0.2]
		\draw[lightgray, fill, rotate around={-45:(2,3)}] (2,3) ellipse (20pt and 60pt);
		\draw[lightgray, fill, rotate around={-45:(10.5,11.5)}] (10.5,11.5) ellipse (20pt and 40pt);
		\plotpermgraph{2,3,4,6,1,8,5,10,7,11,12,9};
		\plotpermbox{0.5}{0.5}{12.5}{12.5};
	\end{tikzpicture}
\quad\quad
	\begin{tikzpicture}[scale=0.2]
		% Intervals:
		\draw[lightgray, fill, rotate around={-45:(1.5,2.5)}] (1.5,2.5) ellipse (20pt and 40pt);
		\draw[lightgray, fill, rotate around={-45:(9.5,10.5)}] (9.5,10.5) ellipse (20pt and 40pt);
		% The permutation:
		\plotpermbox{0.5}{0.5}{11.5}{11.5};
		\plotpermgraph{2,3,5,1,7,4,9,6,10,11,8};
	\end{tikzpicture}
\end{center}
\caption{The permutation on the left is the only member of its antichain (increasing oscillations with both leaves inflated by $123$) that contains the two sum indecomposable children in the center or the sum indecomposable grandchild on the right.}
\label{fig-forced-extra-kids}
\end{figure}

Now consider the antichain formed by inflating both leaves of odd-length increasing oscillations of primary type by $123$. A typical member, of length $13$, in this antichain is shown on the left of Figure~\ref{fig-forced-extra-kids}. This antichain (or a symmetry of it) would have to be considered because the growth rate of the sum closure of its proper downward closure is only approximately $2.34$.

However, the antichain member on the left of Figure~\ref{fig-forced-extra-kids} has two sum indecomposable children of length $12$ and one sum indecomposable grandchild of length $11$ that are not contained in any other member of the antichain. Therefore if we choose to include the antichain member, we must also include these three additional permutations. Of course, we may also choose to take only the grandchild, one child and the grandchild, or both children and the grandchild. Thus we cannot characterize the realizable sequences arising from this antichain in terms of domination of sequences.

Instead, it seems likely one would have to express the realizable sequences as Bevan~\cite{bevan:intervals-of-pe:} did, using the $\beta$ bases of R\'enyi~\cite{renyi:representations:}. To briefly introduce this perspective, note that if the sequence of sum indecomposable permutations in a sum closed class is $(s_n)$ then its growth rate $\gamma$ is the unique real number $\gamma$ such that
\[
	\sum_{n\ge 1} s_n\gamma^{-n}=1,
\]
or in other words, $1$ can be represented in base $\gamma$ by the sequence $(s_n)$ of digits. In the case of sum closed classes lying between $\bigoplus\PropSub(U^o)$ and $\bigoplus\Sub(U^o)$, we are allowed digits $s_{2n}=4$ for $n\ge 3$ and $s_{2n+1}\in\{4,5\}$ for $n\ge 3$. In this viewpoint, we can handle the multitude of choices we have in choosing antichain members, children, and grandchildren in antichains such as that shown in Figure~\ref{fig-forced-extra-kids} by allowing for generalized digits such as $7.21$ (choosing the antichain member, its two children, and its grandchild), $6.21$ (choosing only the two children and the grandchild), $6.11$ (choosing one child and the grandchild), $6.01$ (choosing only the grandchild), and $6$ (choosing none).

\bigskip

\minisec{Acknowledgements}
We are grateful to David Bevan for his many insightful comments on this work. We are also grateful to the referees for their helpful suggestions.

%\bibliographystyle{acm}
%\bibliography{xi}

\end{document}